\documentclass [11pt,a4paper]{article}

\usepackage{amsmath}
\usepackage{amsthm}
\usepackage{color,graphicx}
\usepackage{latexsym}
\usepackage{amssymb}
\usepackage{enumerate}
\usepackage[latin2]{inputenc}
\usepackage{color}
\usepackage{pslatex}
\newtheorem  {theorem}       {Theorem}[section]
\newtheorem  {lemma}         {Lemma}[section]

\newtheorem  {definition}    {Definition}[section]
\newtheorem* {theorem*}      {Theorem}
\newtheorem* {lemma*}        {Lemma}
\newtheorem* {corollary*}    {Corollary}
\newtheorem* {proposition*}  {Proposition}
\newtheorem* {definition*}   {Definition}
\newtheorem* {remark*}       {Remark}
\newtheorem  {remark}        {Remark}
\newtheorem* {remarks*}      {Remarks}
\newtheorem  {conjecture}    {Conjecture}[section]
\newtheorem* {claim*}        {Claim}
\newtheorem  {claim}        {Claim}
\newcommand{\N} {\mathbb N}
\newcommand{\Z} {\mathbb Z}

\newcommand{\R} {\mathbb R}

\newcommand{\Ordo} {{\cal O}}

\newcommand{\ind}{ \mathbb I}
\newcommand{\wt} {\widetilde}

\newcommand{\be}{\begin{equation}}
\newcommand{\ee}{\end{equation}}

\newcommand{\bea}{\begin{eqnarray}}

\newcommand{\eea}{\end{eqnarray}}
\newcommand{\beqs}{\begin{eqnarray*}}
\newcommand{\eeqs}{\end{eqnarray*}}
\newcommand{\beq}{\begin{eqnarray}}
\newcommand{\eeq}{\end{eqnarray}}
\newcommand{\beas}{\begin{eqnarray*}}
\newcommand{\eeas}{\end{eqnarray*}}

\newcommand{\abs}[1]{\left| #1  \right|}

\newcommand{\prob}[1]{\ensuremath{\mathbf{P}\left(#1\right)}}
\newcommand{\expect}[1]{\ensuremath{\mathbf{E}\left(#1\right)}}

\newcommand{\condprob}[2]{\ensuremath{\mathbf{P}\big(#1\,\big|\,#2\big)}}
\newcommand{\condexpect}[2]{\ensuremath{\mathbf{E}\big(#1\,\big|\,#2\big)}}

\newcommand{\cC}{\mathcal{C}}

\newcommand{\cF}{\mathcal{F}}

\newcommand{\cH}{\mathcal{H}}

\newcommand{\cV}{\mathcal{V}}
\newcommand{\cW}{\mathcal{W}}

\newcommand{\deri}[2]{\ensuremath{{#1}'}} % egy ketvaltozos fuggveny,aminek a neve #1 derivaltja a #2 nevu ter-valtozoja szerint
\newcommand{\kderi}[2]{\ensuremath{{#1}''}} %  egy ketvaltozos fuggveny,aminek a neve #1 masodik derivaltja a #2 nevu  ter-valtozoja szerint
\newcommand{\ER} {Erd\H{o}s-R\'enyi }
\newcommand{\conf}[1]{\underline{\mathbf{#1}}}% configuration
\newcommand{\rv} {v}
\newcommand{\myN} {\mathbf{V}} % This is how we denote the set of nonnegative summable sequences
\newcommand{\myNz} {\mathbf{V}^*} %This is how we denote the set of nonnegative sequences with only finitely many positive elements
\newcommand{\phisup} {E_{sup}}
\newcommand{\phiinf} {E_{inf}}
\newcommand{\bt} {T^b} %burning time
\newcommand{\gt} {T^g} %gelation time
\newcommand{\epslam} {\delta(\lambda)} %we need this for the proof of the alternating limit theorem
\newcommand{\mycond} {\text{ } | \text{ }}

\newcommand{\rbspace} {\cW}
\newcommand{\Prb} {\mathbb{P}}

\newcommand{\weak} {\Rightarrow}
\newcommand{\rbeps} {\varepsilon}
\newcommand{\rbt} {\tilde{t}}
\newcommand{\spect} {t^*}

\newcommand{\rbu} {\nu}
\newcommand{\rbtfin} {\overline{t}}
\newcommand{\E} {\mathbb E}

\title
{Mean field frozen percolation}
\author {Balázs Ráth}
\date{}

\begin{document}

\maketitle

%\titlerunning{Short form of title}        % if too long for running head
% The correct dates will be entered by the editor

\maketitle

\begin{abstract}
We define a modification of the \ER random graph process which can
be regarded as the mean field frozen percolation process. We
describe the behavior of the process using differential equations
and investigate their solutions in order to show the
self-organized critical and extremum properties of the critical
frozen percolation model. We prove two limit theorems about the
distribution of the size of the component of a typical frozen vertex.

% \PACS{PACS code1 \and PACS code2 \and more}
% \subclass{MSC code1 \and MSC code2 \and more}
\end{abstract}

\section{Statements}\label{section_statements}

The frozen percolation process on a binary tree was defined by D.
J. Aldous in \cite{Aldousfrozen}: it is a modification of the
percolation process which makes  the
following informal description mathematically rigorous: we only occupy an edge if both
end-vertices are in a finite cluster. The self-organized critical
property of this model manifests in the fact that for $t \geq
\frac12$, which is the critical time of the corresponding
percolation process, a typical finite cluster has the distribution
of a critical percolation cluster.

I. Benjamini and O. Schramm  showed that it is impossible to
define a similar modification of the percolation process on
$\Z^2$. An explanation of this non-existence result can be found
in Section 3. of \cite{BalintRob}.

First we give an informal description of the mean field frozen
percolation process: It is a modification of the \ER random graph
process: Initially we have a (not necessarily empty)
 graph on $\lfloor N\cdot m_0(0) \rfloor$
vertices (one should think about $N$ as being large, but the
initial mass $m_0(0)$ is fixed), and
 between every possible pair of vertices, edges appear with rate
 $\frac{1}{N}$. Simultaneously lightnings strike vertices
 with rate $\lambda(t)\mu(N)$ at time $t$ and when a vertex is struck, the
 fire spreads along the edges and burns the connected component of
 that vertex: that subgraph is removed from the graph, including
 vertices. Thus the number of vertices of the random graph
 decreases with time. The expressions
 "burnt", "frozen", "deleted" and "removed" are treated as synonyms in the
 sequel.

If $\cV_k^N(t)$ denotes the number of vertices contained in
components of size $k$ in the random graph at time $t$, then the
vector-valued stochastic process
$\conf{\cV}(t)=(\cV_1^N(t),\cV_2^N(t),\dots)$ also has the Markov
property (the main advantage of the mean field model is that the graph structure of the connected components has no effect on the evolution of component sizes).
 We are interested in the model when $1 \ll N$.

 Denote by $\N=\{1,2,\dots\}$ and $\N_0=\{0,1,2,\dots \}$.

\begin{definition}\label{def_process}
We fix $m_0(0) \in \R_{+}$.
 The mean field frozen percolation process on $N$
vertices is a continuous time Markov process with state space
\[\Omega_N=\{ \conf{\cV} \in \N_0^{\N} : \sum_{k\geq 1} \cV_k \leq
\lfloor N\cdot m_0(0) \rfloor, \; \forall k\;   \frac{\cV_k}{k} \in \N_0  \}\]
We define the coagulation and deletion operators
 \begin{align}
\conf{\cV}_{k,l}^+ &:= \left\{
\begin{array}{ll}
(\cV_1,\cV_2,\dots, \cV_k-k,\dots,\cV_l-l, \dots,
\cV_{k+l}+k+l,\dots) &
\mbox{ if $k < l$}\\
(\cV_1,\cV_2,\dots, \cV_k-2k,\dots, \cV_{2k}+2k,\dots) & \mbox{ if
$k=l$}
\end{array} \right.\\
\label{def_frozen_operator}
\conf{\cV}_k^- &:= (\cV_1,\dots,\cV_k-k,\dots)
\end{align}
 Let $\lambda:
\R_+ \to \R_+$ be a positive continuous function and $\mu: \N \to
\R_+$.
 The transition rates of the Markov process are
%%%
\begin{align}
\label{coag_rate}
 \lambda(\conf{\cV} \to \conf{\cV}_{k,l}^+) &=\left\{
\begin{array}{ll}
\frac{1}{N} \cdot \cV_k \cdot \cV_l & \mbox{ if $k < l$}\\
\frac{1}{N} \cdot \frac{\cV_k \cdot (\cV_k-k)}{2} & \mbox{ if $k =
l$}
\end{array} \right.
 \\
\label{lightning_rate}
 \lambda(
 \conf{\cV} \to \conf{\cV}_k^-)&=\lambda(t)\cdot \mu(N)\cdot \cV_k
 \end{align}

 Let
 $\rv^N_k(t):=\frac{\cV_k(t)}{N}$ denote the \emph{mass} of components of size $k$ at time $t$.

\end{definition}

The mean field frozen percolation model is closely related to the
mean field forest fire model (discussed in \cite{forestfiPDE}),
the only difference in the definition of the Markov process is
that in the case of the forest fire model, a burnt component of
size $k$ is replaced by $k$ isolated vertices, so that the number
of vertices in the random graph remains unchanged. The two models
both have the self-organized critical property
(and we believe that they are in
the same universality class, which means that the theorems of this paper have analogous "forest fire" versions),
 but the
corresponding partial differential equations have an explicit
solution in the case of the frozen percolation model which enables
us to say more about this model.
\begin{align*}
{\myN} &:= \{ \conf{v} = {\big(v_k\big)}_{k=1}^\infty  :\, \,
v_k \in \R, \, v_k \geq 0  \text{ \ and \ } \sum_{k=1}^\infty v_k < \infty \}\\
{\myNz}&:= \{ \conf{v}  : \, \,  \conf{v} \in \myN,  \, \, \exists
K<+\infty \,\, \forall k \geq K \quad v_k =0 \}
\end{align*}
\begin{definition} \label{def_sub_alt_crit}
 We consider a sequence of mean field frozen percolation processes
with $N \to \infty$, but with the initial state
\[\conf{v}(0)= \left( v_1^N(0),
v_2^N(0),\dots,v_K^N(0),0,0,\dots \right)= \left(\frac{\cV_1^N(0)}{N},
\frac{\cV_2^N(0)}{N},\dots,\frac{\cV_K^N(0)}{N},0,0,\dots \right)
 \in \myNz\]
and the lightning rate function $\lambda(t)$ fixed (independently
of $N$). Such a sequence is called
\begin{itemize}
\item subcritical if $\mu(N) \equiv 1$
\item critical if $\frac{1}{N} \ll \mu(N) \ll 1$
\item alternating if $\mu(N)=\frac{1}{N}$.
\end{itemize}
If $v_k(0)=\ind_{ \{k=1\}} \cdot m_0(0)$ then the initial state is
called \emph{monodisperse}, otherwise it is \emph{polydisperse}.
\end{definition}

We are going to describe the time evolution of the limit object
\begin{equation}\label{limit_object_vague}
\lim_{N \to \infty} v_k^N(t)=v_k(t).
\end{equation}
 We introduce differential equations to characterize the limiting
component size distributions  $v_k(t)$  where $k \in \N$ and $t \in \R_+$. They are modifications
of the Smoluchowski coagulation equation with multiplicative rate kernel:
\begin{align}
\label{flory_ode}
\dot{c}_k(t)&=\frac12 \sum_{l=1}^{k-1} l\cdot (k-l) \cdot c_l(t) \cdot c_{k-l}(t) -c_k(t) \sum_{l=1}^{\infty} l\cdot c_l(0)
 \quad & \text{Flory's model}\\
 \label{stockmayer_ode}
  \dot{c}_k(t)&=\frac12 \sum_{l=1}^{k-1} l\cdot (k-l) \cdot c_l(t) \cdot c_{k-l}(t) -c_k(t) \sum_{l=1}^{\infty} l\cdot c_l(t)
 \quad & \text{Stockmayer's model}
\end{align}
If we let $v_k(t)=k\cdot c_k(t)$ then \eqref{flory_ode} becomes
\begin{equation}\label{smoluchowki_ER}
 \dot{v}_k(t)=\frac{k}{2} \sum_{l=1}^{k-1} v_l(t) v_{k-l}(t) - k\cdot v_k(t)\cdot \sum_{k=1}^{\infty} v_k(0)
 \end{equation}
We are going to use the formulation \eqref{smoluchowki_ER} rather than the classical \eqref{flory_ode}.

The differential equations \eqref{smoluchowki_ER}  describe the time evolution of $\left(v_k(t)\right)_{k=1}^{\infty}$  defined by \eqref{limit_object_vague} for the dynamical
\ER random graph process (see \cite{aldous_smol_survey}). If we only look at the evolution of the component size vector
$\conf{\cV}(t)$ in the dynamical \ER random graph model, we get the Marcus-Lushnikov process (see \cite{Lushnikov}) with multiplicative kernel
 which is
the $\mu(N) \equiv 0$ case of our model (no deletions, only coagulations).

\begin{definition}
If $\big(v_k\big)_{k=1}^{\infty}=\conf{v} \in \myN$ let
\[m_0:=\sum_{k\geq 1} v_k \qquad
 m_1:= \sum_{k \geq 1} k v_k \qquad
m_2:= \sum_{k \geq 1} k^2 v_k \qquad
 m_3:= \sum_{k \geq 1} k^3 v_k\]

\end{definition}
\begin{remark}
Our definition of the moments $m_n$ differs from the convention of
the literature of the Smoluchowski equation by a shift of indices.
\end{remark}

If we define
 \be\label{frozen_evolution_for_N}
w_k^N(t):=\sum_{l=1}^k v_l^N(t) \quad \quad \text{ and } \quad
\quad \Phi^N(t):=\sum_{l\geq1}v_l^N(0) -\sum_{l\geq 1} v_l^N(t)=
m_0^N(0)-m_0^N(t)
\ee
 then for all $k$ the
random function $w_k^N(t)$ is decreasing and $\Phi^N(t)$ (the mass
of burnt vertices) is increasing.

It might happen (e.g. in the case of the  \ER model) that
\[
 \theta(t):= \lim_{k \to \infty} \lim_{N \to \infty} \left(m_0^N(t)-w_k^N(t)\right) \neq
 \lim_{N \to \infty} \lim_{k \to \infty} \left( m_0^N(t)-w_k^N(t)\right) =0.
 \]
In this case the mass missing from the small components is contained in a giant component of mass $0<\theta(t)$.

\begin{definition}\label{def_gelation}
If $\conf{v}(t)$ is a solution of
\eqref{smoluchowki_ER}, we define the gelation time
by \[\gt:= \inf \{t: m_1(t) = +\infty \}.\]
\end{definition}
It is well-known from the theory of the Smoluchowski coagulation
equation that an alternative characterisation of the gelation time
is \[\gt= \inf \{t: m_0(t) < m_0(0) \}.\]

For the solution of \eqref{smoluchowki_ER} the gelation time is $\gt=\frac{1}{m_1(0)}$, the mass of the giant component is
$\theta(t)=m_0(0)-m_0(t)$. $\conf{v}(t)$ undergoes a phase transition:
\begin{itemize}

\item
For $0\le t< \gt$ the system is subcritical:
$\theta(t)=0$ and $k\mapsto v_k(t)$ decay exponentially
with $k$.

\item
For $\gt <t$ the system is supercritical:
$\theta(t)>0$ and $k\mapsto v_k(t)$ decay exponentially
with $k$.
  Further on: $t\mapsto \theta(t) $ is smooth and
strictly increasing with  $\lim_{t\to\infty} \theta(t)=m_0(0)$.

\item
Finally, at  $t= \gt$ the system is critical:
$\theta(t)=0$ and
\begin{equation}
\label{erdos_renyi_critical_exponent}
 \sum_{k=K}^{\infty} v_k(\gt)\asymp
K^{-1/2} \quad \text{as $\quad K\to\infty$}.
\end{equation}

\end{itemize}

Our aim is to understand in similar terms the asymptotic behavior
of the system when, beside the Erd\H os-R\'enyi coagulation
mechanism, deletions due to lightnings also take place.

\begin{definition}\label{definition_general_frozen_eqs}
 We say that
$\conf{v}(t)= \left( v_k(t) \right)_{k=1}^{\infty} \in \myN$
solves the general frozen percolation equation on $\lbrack 0,
T\rbrack$ with initial condition $\conf{v}(0) \in \myNz$, a
continuous nonnegative rate function $\lambda: \R_+ \to \R_+$ and
 control function $\Phi:
\R_+ \to \R_+$ if
 \be
\label{control_function_Phi_increasing_bounded}
 \forall \; 0 \leq s \leq t \leq T \quad 0 \leq \Phi(0)\leq \Phi(s)
 \leq \Phi(t) < m_0(0)
 \ee
and for all $k=1,2,\dots$ the equations \be
\label{general_frozen_eqs} v_k(t)=v_k(0)+\int_{0}^t \frac{k}{2}
\sum_{l=1}^{k-1} v_l(s) v_{k-l}(s) -k v_k(s) \left(
(m_0(0)-\Phi(s))+\lambda(s)\right)ds
 \ee
 and the inequality
 \be \label{theta_def_nonnegative_general_frozen_eq}
 \forall t \quad 0 \leq
\theta(t):=m_0(0)-m_0(t)-\Phi(t)\ee is satisfied.
\end{definition}
It is easy to see by induction that the absolutely continuous
functions $v_1(t),v_2(t),\dots$ are completely determined by
\eqref{general_frozen_eqs}, the
initial condition $\conf{v}(0)$ and the functions $\lambda$ and
$\Phi$. The only reason why we do not write
\begin{equation}\label{general_frozen_ODE}
 \dot{v}_k(t)=\frac{k}{2}
\sum_{l=1}^{k-1} v_l(t) v_{k-l}(t) -k v_k(t) \left(
(m_0(0)-\Phi(t))+\lambda(t)\right)
 \end{equation}
instead of \eqref{general_frozen_eqs} is that the increasing
function $\Phi(t)$ might have jumps.

There are three versions of the general frozen percolation
equation corresponding to the three regimes on Definition
\ref{def_sub_alt_crit}:

%\begin{definition} \label{def_sub_crit_smol_equations}

\begin{itemize}
\item The \emph{subcritical system of integral  equations}
are \eqref{general_frozen_eqs} with the extra conditions $\forall
t \; 0< \lambda_{inf} \leq \lambda(t)$ and
\be\label{subcrit_Phi_def}
 \Phi(t) \equiv m_0(0)-m_0(t).\ee
That is $\theta(t)\equiv 0$ by \eqref{theta_def_nonnegative_general_frozen_eq} (no giant components appear due to frequent lightnings)
 and the equations take on the form
 \be \label{subcrit_smol_frozen_integal_eq} v_k(t)= v_k(0)+
\int_0^t \frac{k}{2} \sum_{l=1}^{k-1} v_l(s)v_{k-l}(s) -k \cdot
v_k(s)m_0(s) -\lambda(s) k \cdot v_k(s) ds \ee
The term $-\lambda(s) k \cdot v_k(s)$ indicates that in the subcritical regime even small components
are burnt with a rate proportional to their sizes and $\lambda(s)$.

% \be \label{subcrit_smol_frozen_eq} \dot{v}_k(t)=\frac{k}{2}
%\sum_{l=1}^{k-1} v_l(t)v_{k-l}(t) -k \cdot v_k(t)m_0(t)
%-\lambda(t) k \cdot v_k(t)
% \ee

\item The \emph{critical equations} are \eqref{general_frozen_eqs}
 with the extra
conditions $\lambda(t)\equiv 0$ and \eqref{subcrit_Phi_def}:
 \be
\label{crit_smol_frozen_integral_eq}  v_k(t)= v_k(0)+ \int_0^t
\frac{k}{2} \sum_{l=1}^{k-1} v_l(s)v_{k-l}(s) -k \cdot
v_k(s)m_0(s) ds \ee
$\lambda(t)\equiv 0$ indicates that in the critical regime lightnings are not
frequent enough to do any harm to small components, but \eqref{subcrit_Phi_def} indicates that
they are frequent enough to keep the mass of the giant component at zero.

\item Let $0=\bt_0 < \bt_1 < \bt_2<\dots$ be a sequence with no
accumulation points. Let \be \label{def_of_M_counting} M(t):=\max
\lbrace i: \bt_i < t \rbrace \ee   $\conf{v}(t)$ solves the
\emph{alternating equations} with burning times $\bt_1,\bt_2,\dots$ if
\be \label{alter_smol_frozen_eq} \dot{v}_k(t)=\frac{k}{2}
\sum_{l=1}^{k-1} v_l(t)v_{k-l}(t) -k \cdot v_k(t)m_0(\bt_{M(t)})
 \ee
\end{itemize}

Mind the difference between \eqref{smoluchowki_ER} and \eqref{crit_smol_frozen_integral_eq}:
in the case of the \ER model the small components are allowed to coagulate with the giant component (which is of size
$\theta(t)=m_0(0)-m_0(t)$ by $\Phi(t) \equiv 0$ and \eqref{theta_def_nonnegative_general_frozen_eq}),
 but in the case of the frozen percolation model the giant components are removed at the time of their birth.
Using the terminology of the theory of Smoluchowski coagulation equations we might say that in the case of
\eqref{smoluchowki_ER} the gel and
the sol do react in the post-gel phase (Flory's model, \eqref{flory_ode}), but in the case of \eqref{crit_smol_frozen_integral_eq} they do not
react (Stockmayer's model, \eqref{stockmayer_ode}).
Nevertheless, for $t \leq \gt$ the solutions of \eqref{smoluchowki_ER} and \eqref{crit_smol_frozen_integral_eq} are identical
since $m_0(t)=m_0(0)$ in this regime.

The intuitive meaning of \eqref{alter_smol_frozen_eq} is that giant components are removed
from the system at the burning times.

Thus \eqref{alter_smol_frozen_eq} is \eqref{general_frozen_eqs}
with
\begin{align}\label{alt_theta_def}
 \theta(t)&=m_0(\bt_{M(t)})-m_0(t)\\
\label{alter_Phi_reformulations} \Phi(t)&=m_0(0)-m_0(\bt_{M(t)})=
 m_0(0)-m_0(t)-\theta(t)=
\sum_{j=1}^{M(t)}
  \theta(\bt_j)
  \end{align}
  Both $\theta(t)$ and $\Phi(t)$ are left-continuous functions of
  $t$.

Note that in the case of the (sub)critical frozen percolation
equations (\eqref{subcrit_smol_frozen_integal_eq} and \eqref{crit_smol_frozen_integral_eq}) the fact that $\Phi(t)$ is an increasing function automatically
 follows by \eqref{subcrit_Phi_def}:
\begin{multline*}
\Phi(t)-\Phi(s)=m_0(s)-m_0(t)=\\
\sum_{k=1}^{\infty} \int_s^t -\frac{k}{2} \sum_{l=1}^{k-1} v_l(u)v_{k-l}(u) +k \cdot
v_k(u)m_0(u) +\lambda(u)\cdot k \cdot v_k(u)\, du=\\
\lim_{N \to \infty} \int_s^t \sum_{k=1}^N \sum_{l=N-k+1}^{\infty}
k \cdot v_k(u)v_l(u) +\lambda(u) \cdot k \cdot v_k(u) \, du \geq 0
\end{multline*}

\begin{theorem}
$ $

\begin{itemize}
\item For any $\conf{v}(0) \in \myNz$ and $0<\lambda_{inf} \leq
\lambda(t)$ the equations \eqref{subcrit_smol_frozen_integal_eq}
have a unique solution.
\item For any $\conf{v}(0) \in \myNz$ the equations \eqref{crit_smol_frozen_integral_eq}
have a unique solution.
\item For any $\conf{v}(0) \in \myNz$ and
any sequence of burning times the equations
\eqref{alter_smol_frozen_eq} have a unique solution.
\end{itemize}
\end{theorem}
We prove this theorem in
 Section \ref{section_wellposed}.

\begin{definition}\label{def_random_alternating_eqs}
The solution of the random alternating equations with rate
function $\lambda: \R_+ \to \R_+$
 is a $\myN$-valued continuous-time Markov process:
$\conf{v}(t)$ evolves deterministically, driven by the equations
\eqref{alter_smol_frozen_eq}, but the sequence of burning times
$\bt_1,\bt_2,\dots$ is random: \be \label{ramdom_alt_burning_time}
\lim_{dt \to 0} \frac{1}{dt} \condprob{ t \leq \bt_{M(t)+1} \leq
t+dt}{\cF_t} =\lambda(t) \theta(t)
 \ee
where $\cF_t$ is the natural filtration generated by the process.
\end{definition}
In plain words: a lightning strikes and burns the giant component
 with rate proportional to its size and $\lambda(t)$.

%Note that if  $\theta(\bt_i)=0$ for a burning time of the solution
%of \eqref{alter_smol_frozen_eq}  then there is no giant to burn,
%so ``nothing happens'' and we could as well remove $\bt_i$ from
%the sequence of burning times without changing the solution.

%We do not prove the following conjectures in this paper, but we
%believe that they can be proven by the methods introduced in
%\cite{forestfiPDE}.

\begin{definition}\label{def_frozen_percolation_evolution}

\begin{align*}
 \cW & :=\left\{ \left(w_k \right)_{k=1}^{\infty} \, \, : \, \,
0\leq w_1 \leq w_2 \leq \dots <+\infty \right\} \\
 \cW^* & := \left\{ \left(w_k \right)_{k=1}^{\infty} \in \cW \, \, :
\, \, \exists K<+\infty \; \; \forall k  \geq K \,\; w_k=w_K
\right\}
\end{align*}
If $\conf{w} \in \cW$ denote by $m_0:=\sup_k w_k$.

We say that $\left( \left(w_k(\cdot) \right)_{k=1}^{\infty},
\Phi(\cdot) \right)$ is a frozen percolation evolution on $\lbrack
0, T \rbrack$ with initial condition $\left( w_k(0)
\right)_{k=1}^{\infty}=\conf{w} \in \cW^*$, or briefly \[\left(
\left(w_k(\cdot) \right)_{k=1}^{\infty}, \Phi(\cdot) \right) \in
\rbspace_{\conf{w}} \lbrack 0, T \rbrack \]
 if for all
$0\leq t\leq T$ we have $\left(w_k(t)\right)_{k=1}^{\infty} \in
\cW$, for all
 $k$ the
functions $w_k: \lbrack 0, T \rbrack \to \lbrack 0, m_0(0)
\rbrack$ are left-continuous and decreasing, $\Phi: \lbrack 0, T
\rbrack \to \lbrack 0, m_0(0) \rbrack$ is left continuous and
increasing with initial condition $\Phi(0)=0$, moreover for all $t
\leq T$ we have \eqref{theta_def_nonnegative_general_frozen_eq}.

We define convergence on the space $\rbspace_{\conf{w}} \lbrack 0,
T \rbrack$:
 \[\left( \left(w_k^n(\cdot) \right)_{k=1}^{\infty},
\Phi^n(\cdot) \right) \to \left( \left(w_k(\cdot)
\right)_{k=1}^{\infty}, \Phi(\cdot) \right)\]
 as $n \to \infty$ if
for all $k$ we have $w_k^n(t) \to w_k(t)$ for all $t$ which is a
point of continuity of $w_k$ and $\Phi^n(t) \to \Phi(t)$ for all
$t$ which is a point of continuity of $\Phi$.
\end{definition}
With this topology the space $\rbspace_{\conf{w}} \lbrack 0, T
\rbrack$ is metrizable, complete and compact.

 From the frozen
percolation process of Definition \ref{def_process}. one gets a
random element of $\rbspace_{\conf{w}} \lbrack 0, T \rbrack$ by
\eqref{frozen_evolution_for_N}. Denote the probability measure on
$\rbspace_{\conf{w}} \lbrack 0, T \rbrack$ corresponding to the
process by $\Prb_N$.

It is easy to check that $\left( \left(w_k(\cdot)
\right)_{k=1}^{\infty}, \Phi(\cdot) \right) \in
\rbspace_{\conf{w}} \lbrack 0, T \rbrack$ where
$w_k(t)=\sum_{l=1}^k v_l(t)$ and $\conf{v}(t)$ is a solution of
the general frozen percolation equation \eqref{control_function_Phi_increasing_bounded} \&
 \eqref{general_frozen_eqs} \& \eqref{theta_def_nonnegative_general_frozen_eq}.

\begin{theorem}\label{thm_convergence_of_process_to_eqn}
We consider a sequence of frozen percolation processes (see
Definition \ref{def_process}) with initial state $\conf{v}^N (0)=
\conf{v}(0) \in \myNz$ and $\lambda(t)$ positive and continuous.
Define $w_k^N(t)$ and $\Phi^N(t)$ as in
\eqref{frozen_evolution_for_N}. Denote the probability measure on
$\rbspace_{\conf{w}} \lbrack 0, T \rbrack$ corresponding to the
process by $\Prb_N$.

 Then $\Prb_N$ converges with respect to the weak
 convergence of probability measures on the polish space
  $\rbspace_{\conf{w}} \lbrack 0, T \rbrack$ to a limiting measure
  $\Prb$, which depends on the decay rate of $\mu(N)$ in the
  following way:
\begin{itemize}
\item If $\mu(N)\equiv 1$ then $\Prb$ is concentrated on the unique solution of
\eqref{subcrit_smol_frozen_integal_eq} with rate function
$\lambda(t)$.
\item If $\frac{1}{N} \ll \mu(N) \ll 1$ then $\Prb$ is concentrated
 on the unique solution of
\eqref{crit_smol_frozen_integral_eq}.
\item If $\mu(N)= \frac{1}{N}$ then $\Prb$ is the law
of the solution of the random alternating equation (see Definition
\ref{def_random_alternating_eqs}) with rate function $\lambda(t)$.
\end{itemize}
\end{theorem}
We prove the $\mu(N)\equiv 1$ and the $\frac{1}{N} \ll \mu(N) \ll
1$ part of this theorem in Section \ref{proof_of_convergence}. In
fact, these proofs are almost identical to the corresponding
convergence results of \cite{forestfiPDE}, but we present them
here as well for the sake of completeness.

 We
omit the proof of the $\mu(N)= \frac{1}{N}$ part of Theorem
\ref{thm_convergence_of_process_to_eqn}., but we believe that the
methods introduced in Section \ref{proof_of_convergence}. can be
easily generalized for this case as well.

%of the integral/differential equations
% \eqref{subcrit_smol_frozen_integal_eq},
%\eqref{crit_smol_frozen_integral_eq} and
%\eqref{alter_smol_frozen_eq} in the Lemmata
%\ref{alter_eq_well_posed}., \ref{subcrit_integral_eq_uniqueness}.,
%\ref{crit_eq_existence}., and \ref{subcritical_existence_lemma}.

If we formally substitute $\lambda(t) \equiv 0$ into
\eqref{subcrit_smol_frozen_integal_eq} or $\bt_{M(t)} \equiv t$
into \eqref{alter_smol_frozen_eq}, we get
\eqref{crit_smol_frozen_integral_eq}. Rigorously:

\begin{theorem}\label{subcrit_converges_to_crit}
Let $\left(\conf{v}^n(t)\right)_{n=1}^{\infty}$  be a sequence of
solutions of \eqref{subcrit_smol_frozen_integal_eq} with the same
initial condition $\conf{v}(0) \in \myNz$ where $\lambda_n(t) \to
0$ uniformly as $n \to \infty$. Then for all $t$ and $k$ $\lim_{n
\to \infty} v_k^n(t)=v_k(t)$ where $\conf{v}(t)$ is the solution
of \eqref{crit_smol_frozen_integral_eq} with the same initial
data. $\lim_{n \to \infty} \Phi_n(t)=\Phi(t)$ uniformly on
$\lbrack 0, \infty)$.
\end{theorem}
In plain words: if the rate of lightning is very small in the subcritical equations,
 then the solution is similar to that of the critical equation.
We prove this theorem in Section \ref{section_proof_of_sub}.

\begin{theorem}\label{alternating_converges_to_crit}
Let $\left(\conf{v}^n(t)\right)_{n=1}^{\infty}$  be a sequence of
solutions of \eqref{alter_smol_frozen_eq} with the same initial
condition $\conf{v}(0)$ where the sequence of burning times
satisfy \[\lim_{n \to \infty} \sup_i \{
 \bt_{i+1}(n)-\bt_{i}(n)
 \}=0.\] Then for all $t$ and $k$ $ \lim_{n \to \infty}
v_k^n(t)=v_k(t)$ where $\conf{v}(t)$ is the solution of
\eqref{crit_smol_frozen_integral_eq} with the same initial data.
$\lim_{n \to \infty} \Phi_n(t)=\Phi(t)$ uniformly on $\lbrack 0,
\infty)$.
\end{theorem}
In plain words: if the burning times of the alternating equations are very frequent,
then the solution is similar to that of the critical equation.
We prove this theorem in Section \ref{section_proof_of_alt}.

 The solution of \eqref{crit_smol_frozen_integral_eq} has
the self-organized critical property: for all $\gt \leq t$
 it has the power-law decay of \eqref{erdos_renyi_critical_exponent}:
\begin{theorem}\label{tauberian_smol}
If $\conf{v}(t)$ is a solution of
\eqref{crit_smol_frozen_integral_eq} with initial condition
$\conf{v}(0) \in \myNz$, then $\gt=\frac{1}{m_1(0)}$,
$\Phi(t)=\int_{\gt}^t \varphi_{crit}(s)ds$ where $\varphi_{crit}:
\lbrack \gt,+\infty) \to R_+ $ is positive and continuous, and for
all $t\geq \gt$ we have \be \label{powerlaw_decay} \lim_{K \to
\infty} K^{\frac12}\sum_{k=K}^{\infty} v_k(t)=\sqrt{\frac{2
\varphi_{crit}(t)}{\pi}}.\ee
\end{theorem}

\begin{definition}
Let $x^*(t):= \inf \lbrace x: \sum_{k=1}^{\infty} v_k(t)e^{-kx} <
 +\infty \rbrace$
\end{definition}

The solutions of our equations have a remarkable rigidity property:

\begin{theorem}\label{exp_tilt_rigidity}
If $\conf{v}(t)$ is the solution of
 \eqref{subcrit_smol_frozen_integal_eq} or \eqref{alter_smol_frozen_eq} and
 $\tilde{\conf{v}}(t)$ is the solution of
\eqref{crit_smol_frozen_integral_eq} with the same initial
condition, then for all $t \geq \gt$ and $k \geq 1$ we have
\[\tilde{v}_k(t)=v_k(t)e^{-k x^*(t)}.\]
\end{theorem}

The solution of \eqref{crit_smol_frozen_integral_eq} with
monodisperse initial condition is well-known (see e.g.
\cite{ZiffErnstHendriks}) and explicit:
\begin{claim}\label{selfsimilarity}
If  $\conf{v}(t)$ is the solution of
\eqref{crit_smol_frozen_integral_eq} with $v_k(0)=\ind_{\lbrace
k=1 \rbrace}\cdot m_0(0)$ then for $t \geq \gt
=\frac{1}{m_1(0)}=\frac{1}{m_0(0)}$ we have
\begin{equation}\label{borel}
v_k(t)=\frac{1}{t}
\frac{k^{k-1}}{k!}e^{-k}.
\end{equation}
\end{claim}
That is, for all $\gt \leq t$ in the $N \to \infty$ limit, the component size of a uniformly chosen (unburnt) vertex
in the critical frozen percolation model has Borel distribution, which
is the same as that of a vertex in the \ER graph at $t=\gt$. The Borel distribution ( $\left( v_k(1) \right)_{k=1}^{\infty}$ in \eqref{borel})
 is the distribution of the size of a critical
Galton-Watson tree with $POI(1)$ offspring distribution (see \cite{aldous_smol_survey}).

The same self-similarity phenomenon can be observed in Aldous' frozen percolation model (see \cite{Aldousfrozen})
 on the binary tree: for $t \geq
\frac12$, which is the critical time of the
percolation process on the binary tree, a typical finite cluster has the distribution
of a critical percolation cluster.

 The solutions started from a polydisperse initial state are asymptotically self-similar:
\begin{theorem}\label{asymptotic_selfsimilar}
If  $\conf{v}(t)$ is the solution of the critical equation
\eqref{crit_smol_frozen_integral_eq} with $\conf{v}(0) \in \myNz$,
and $v_1(0)>0$ then
\begin{equation}\label{asymp_selfsimilar_formula}
\lim_{t \to \infty}\; t\cdot v_k(t) =
\frac{k^{k-1}}{k!}e^{-k}\quad \text{ and }\quad \lim_{t \to
\infty}\; t\cdot m_0(t) =1.
\end{equation}
\end{theorem}

 Theorems \ref{tauberian_smol}., \ref{exp_tilt_rigidity}. and
\ref{asymptotic_selfsimilar}. are proved in Section
\ref{section_properties_of_frozen_eqs} using the method of Laplace transforms, which is classical for the Smoluchowski equation
with multiplicative kernel. The results \eqref{asymp_selfsimilar_formula} and $\frac{v_k(t)}{k}=c_k(t) \asymp k^{-5/2}$
(which is a variant of \eqref{powerlaw_decay}) are already present in \cite{ZiffErnstHendriks}, but we believe that our approach
based on the notion of the \emph{critical core} of $\conf{v}(t)$  (defined in Section \ref{section_def_transf})
 gives new insight into these results about the solution of \eqref{crit_smol_frozen_integral_eq}.

In the frozen percolation model on the binary tree, components are frozen (i.e. removed
from the system) when their size becomes infinite. The question may arise:

$ $

\emph{What is the typical size of a frozen component
 in the mean field process of Definition \ref{def_process}?}

$ $

  In order to precisely formulate this question recall \eqref{def_frozen_operator} and
 let
 \[ \Phi^N([t_1,t_2],k):=\frac{k}{n} \cdot \abs{ \left\{ t \in [t_1,t_2]\;:\; \conf{\cV}(t_+)= \conf{\cV}_k^-(t_-)   \right\} }. \]
  Thus $\Phi^N([t_1,t_2],k)$ is the mass of burnt components of size $k$ from $t_1$ to $t_2$. We have
  \[\sum_{k \geq 1} \Phi^N([t_1,t_2],k)= \Phi^N(t_2)-\Phi^N(t_1)=:\Phi([t_1,t_2])\]
  Thus  $p^N_k[t_1,t_2] :=  \frac{\Phi^N([t_1,t_2],k)}{\Phi^N([t_1,t_2])}$, $k=1,2,\dots$ is a random probability distribution for all $N$ and $t_1<t_2$.

  Denote by $\abs{\cC_{max}^N(t)}$  the size of the largest component at time $t$.
\begin{conjecture}\label{torottvonal_conjecture}
If $\mu(N)=N^{-\alpha}$ in a critical sequence of frozen percolation processes (see Definitions \ref{def_process} and \ref{def_sub_alt_crit}), where
$0<\alpha<1$, and if we define
\begin{equation}\label{torott}
 \beta(\alpha):= \begin{cases} 2\alpha & \text{ if } \alpha \leq \frac13 \\
\frac{\alpha+1}{2} & \text{ if } \alpha \geq \frac13 \end{cases}
\end{equation}
then for every $\gt < t$ we have
\begin{align}
\label{m1_and_alpha} \lim_{N \to \infty} \frac{ \log \left( \expect{ m_1^N(t)} \right)}{\log(N)}\; &= \; \alpha \\
\label{m2_m1_beta} \lim_{N \to \infty} \frac{ \log \left( \expect{ m_2^N(t)} \right) - \log \left( \expect{ m_1^N(t)} \right)}{\log(N)}\; &= \;\beta(\alpha)\\
\lim_{N \to \infty} \frac{ \log \left( \expect{ \abs{\cC_{max}^N(t)}} \right)}{\log(N)}\; &= \;\beta(\alpha)
\end{align}
Moreover for every $\conf{v}(0)$, $\gt <t_1 <t_2$ and $\alpha$
there exists a non-defective probability distribution function $F: (0, \infty) \to (0,1)$, $\lim_{x \to 0_+} F(x)=0$, $\lim_{x \to \infty} F(x)=1$  such that for all $x \in \R_+$ we have
\begin{equation}\label{sejtes}
 \lim_{N \to \infty} \sum_{k \geq 1} \ind [\; k \leq  x  N^{\beta(\alpha)}  \; ] \cdot p^N_k[t_1,t_2]\;  = \; F(x)
 \end{equation}
\end{conjecture}
In plain words we might say that after gelation the typical component size of a frozen vertex and the size of the largest component
 is of order $N^{\beta(\alpha)}$. This conjecture is supported by heuristic arguments, computer simulations and Theorems \ref{subcrit_limit_thm_statement} and \ref{alter_limit_thm} below. For $0<\alpha<\frac13$ the model is conjectured to behave similarly to the subcritical case described in Theorem  \ref{subcrit_limit_thm_statement}, whereas for $\frac13 < \alpha<1$ it is conjectured to behave similarly to the alternating case described in
 Theorem \ref{alter_limit_thm}. Note that $\beta(\frac13)=\frac23$ and $N^\frac23$  is the order of the size of the largest component in the critical \ER random graph.

\begin{theorem}\label{subcrit_limit_thm_statement}
If $\conf{v}^{\lambda}(t)$ is the solution of
\eqref{subcrit_smol_frozen_integal_eq} with rate function
$\lambda(t) \equiv \lambda$ and $\conf{v}^{\lambda}(0)=
\conf{v}(0) \in \myNz$ then
  there is a constant $C$ that depends only on
the initial data and $T$ such that for all $0<\lambda \leq 1$ and
 $\frac{1}{m_1(0)}<t \leq T$
 we have
\be\label{subcrit_control_almost_crit}
\abs{\varphi_{\lambda}(t)-\varphi_{crit}(t)} \leq C\lambda \ee
where
\begin{equation} \label{subcrit_phi_formula}
\frac{d}{dt}\Phi_{\lambda}(t)=\varphi_{\lambda}(t)=\lambda
m_1^{\lambda}(t).
\end{equation}

Moreover if we define the random variable $Y_{\lambda}(t)$ to have
distribution \[\prob{Y_{\lambda}(t)=k}= \frac{\lambda \cdot k \cdot
v_k^{\lambda}(t)}{\varphi_{\lambda}(t)}= \frac{k\cdot
v_k^{\lambda}(t)}{m_1^{\lambda}(t)}\] then \be
\label{subcrit_limit_thm_distribution_formulation} \lim_{\lambda
\to 0} \prob{\frac{\lambda^2}{2 \varphi_{crit}(t)}
Y_{\lambda}(t)<x}= \int_0^x \frac{1}{\sqrt{\pi}}
\frac{1}{\sqrt{y}} e^{-y}dy \ee
\end{theorem}
In plain words: for any $t> \gt$  the distribution of the size-biased sample from the component-size distribution
$\conf{v}^{\lambda}(t)$ rescaled by $\lambda^{-2}$ converges in distribution to a $\Gamma(\frac12, 1)$ distribution as $\lambda \to 0$.
We prove this theorem in Section \ref{section_proof_of_alt}.

The relevance of Theorem \ref{subcrit_limit_thm_statement} to Conjecture \ref{torottvonal_conjecture} is the following:
 if we consider a sequence of subcritical frozen percolation models (see Definition \ref{def_sub_alt_crit}) with $\lambda(t)\equiv \lambda$ then by
Theorem \ref{thm_convergence_of_process_to_eqn} we get
\begin{multline*} \lim_{dt \to 0} \lim_{N \to \infty} p^N_k[t,t+dt] = \lim_{dt \to 0} \frac{\Phi_{\lambda}([t,t+dt],k)}{\Phi_{\lambda}([t,t+dt])}=\\
\lim_{dt \to 0} \frac{\int_t^{t+dt} \lambda \cdot k \cdot v^{\lambda}_k(s) \,ds }{ \int_t^{t+dt} \sum_{l=1}^{\infty} \lambda \cdot l \cdot v^{\lambda}_l(s)}=\frac{k \cdot v_k^{\lambda}(t) }{m_1^{\lambda}(t)}=\prob{Y_{\lambda}(t)=k}
\end{multline*}
If we let $\lambda \to 0$ then by \eqref{subcrit_control_almost_crit} and \eqref{subcrit_phi_formula} we get $m_1^{\lambda}(t) \asymp \lambda^{-1}$ which is a "subcritical" version of
\eqref{m1_and_alpha},  $\frac{m_2^\lambda(t)}{m_1^\lambda (t)} = \expect{Y_{\lambda}(t)}   \asymp \lambda^{-2}$ corresponds to $\beta(\alpha)=2 \alpha$ in \eqref{m2_m1_beta}, and \eqref{subcrit_limit_thm_distribution_formulation} is a  version of \eqref{sejtes}.

\begin{theorem}\label{alter_limit_thm}
Let $\conf{v}^{\lambda}(t)$ denote the solution of the random
alternating equations (see Definition
\ref{def_random_alternating_eqs}.) with a constant rate function
$\lambda(t)\equiv \lambda$.

Let $\epslam$ be a function satisfying $\lambda^{-\frac12} \ll
\epslam \ll 1 $ as $ \lambda \to \infty$.

 Recalling
\eqref{def_of_M_counting} and \eqref{alt_theta_def} let
\[\Phi_{\lambda}(t,x):=\sum_{j=1}^{M(t)} \theta^{\lambda}(\bt_j) \ind
\lbrack \theta^{\lambda}(\bt_j) > x \rbrack \] be the random mass
of frozen giants of size at least $x$. Then \be
\label{alter_limit_thm_statement_formula}\lim_{\lambda \to \infty}
\frac{\Phi_{\lambda}\left(t+\epslam,2\sqrt{\frac{\varphi_{crit}(t)}{\lambda}}
x
 \right)-
\Phi_{\lambda} \left( t,2\sqrt{\frac{\varphi_{crit}(t)}{\lambda}}
x \right)}{\epslam \varphi_{crit}(t)}= \int_x^{\infty}
\frac{4}{\sqrt{\pi}} y^2 e^{-y^2}dy \ee in probability.
\end{theorem}
We prove this theorem in Section \ref{section_proof_of_alt}.

The heuristic meaning of this theorem is the following:
 if we pick a vertex uniformly from all vertices that were frozen between $t$
and $t+\epslam$ and denote the mass of the giant component of that
vertex by $Z_{\lambda}(t)$, then the distribution of $\frac12
\sqrt{\frac{\lambda}{\varphi_{crit}(t)}} Z_{\lambda}(t)$ converges
to a size-biased Rayleigh distribution (see Definition
\ref{Rayleigh}) as $\lambda \to \infty$.
Thus the typical mass of a frozen giant is of
order $\lambda^{-\frac12}$, which suggests that if $\mu(N)=\frac{N^{\varepsilon}}{N}$ (that is $\alpha=1-\varepsilon$ in Conjecture
\ref{torottvonal_conjecture})
then the typical size of a frozen component is of order $\left(N^{\varepsilon}\right)^{-\frac12} \cdot N= N^{1-\frac12 \varepsilon}$, that is
$\beta(\alpha)=\frac{\alpha+1}{2}$. \eqref{alter_limit_thm_statement_formula} is the "alternating" version of \eqref{sejtes}.

$ $

 The critical frozen percolation model has an extremum property
compared to the subcritical and alternating models (see Definition
\ref{def_sub_alt_crit}): if each burnt/frozen vertex produces
profit at a rate $\frac{1}{N} \$ $ per time unit after it has been
frozen, but each lightning (even the ones hitting burnt vertices)
costs $\frac{1}{N \cdot m_0(0)} \$ $, then asymptotically (as $N
\to \infty$)
 the
critical model is the best choice if we want to maximize our
profit on $\lbrack 0, T \rbrack$. We reformulate this extremum
principle in terms of the differential equations
\eqref{subcrit_smol_frozen_integal_eq},
\eqref{crit_smol_frozen_integral_eq},
\eqref{alter_smol_frozen_eq}.

 The asymptotic value of our profit produced by burnt vertices as
$N \to \infty$ is $\int_0^T \Phi(t)dt$ according to Theorem
\ref{thm_convergence_of_process_to_eqn}.
 The asymptotic cost of lightnings is $\int_0^T
\lambda(t)dt$ for the solution of
\eqref{subcrit_smol_frozen_integal_eq}, but it is zero for
\eqref{crit_smol_frozen_integral_eq} and
\eqref{alter_smol_frozen_eq}, since the price we have to pay for
the lightnings vanishes in the case of critical and alternating
models as $N \to \infty$.
% The following theorem uses the
%well-posedness of the equations of Definition
%\ref{def_sub_crit_smol_equations}: once the initial data is fixed,
%the value of $\Phi(t)$ on $\lbrack 0, T \rbrack$ is determined by
%the value of $\lambda(t)$ on $\lbrack 0, T \rbrack$ in the case of
%(\ref{subcrit_smol_frozen_integal_eq}).

\begin{theorem}\label{extremum_theorem}
We fix  $\conf{v}(0) \in \myNz$. Let  $\conf{v}^{crit}(t)$ denote
the solution of \eqref{crit_smol_frozen_integral_eq} with initial
condition $\conf{v}(0)$  and let  $\conf{v}^{sub}$ denote the
solution of \eqref{subcrit_smol_frozen_integal_eq} with lightning
rate function $\lambda(t)$ and
 the same initial condition. Then for any $T>0$
\be \label{subcrit_functional_extremum}
 \int_0^T \Phi^{sub}(t)dt-\int_0^T \lambda(t)\, dt \leq \int_0^T \Phi^{crit}(t)dt-\int_0^T 0\, dt
\ee

If  $\conf{v}^{alt}(t)$ denotes the solution of
\eqref{alter_smol_frozen_eq} with an arbitrary sequence of burning
times and initial condition $\conf{v}(0)$ then \be
\label{alter_functional_extremum} \int_0^T \Phi^{alt}(t)dt \leq
\int_0^T \Phi^{crit}(t)dt \ee

\end{theorem}
\begin{remark}\label{sharp}
Let $T>\gt=\frac{1}{m_1(0)}$ and $\varepsilon>0$. For a suitable choice of $\lambda(t)$ we have
\be \label{subcritical_functional_sharp}
 \int_0^T \Phi^{sub}(t)dt-(1-\varepsilon)\int_0^T \lambda(t)\, dt > \int_0^T \Phi^{crit}(t)dt-(1-\varepsilon)\int_0^T 0\, dt
\ee
For a suitable choice of burning times
\be  \label{alternating_functional_sharp}
\int_0^T \Phi^{alt}(t)dt+ \varepsilon \Phi^{alt}(T) > \int_0^T \Phi^{crit}(t)dt+ \varepsilon \Phi^{crit}(T)
\ee
\end{remark}
The idea that the critical \emph{forest fire} model solves a
variational problem is already present in \cite{DrosselSchwabl}.

\section{Definitions, Transformations}\label{section_def_transf}

We consider a solution of the general frozen percolation equation
(see Definition \ref{definition_general_frozen_eqs}.).

% \be
%\label{subcrit_smol_frozen_integal_eq} v_k(t)= v_k(0)+ \int_0^t
%\frac{k}{2} \sum_{l=1}^{k-1} v_l(s)v_{k-l}(s) -k \cdot
%v_k(s)m_0(s) -\lambda(s) k \cdot v_k(s) ds \ee

Denote the Laplace transform (generating function) of
$\conf{v}(t)$ by
\begin{equation}\label{gfdeffff}
V(t,x):=\sum_{k=1}^{\infty} v_k(t)e^{-kx}
\end{equation}
 for
$x>0$. Then $V(t,0)=V(t,0_+)=m_0(t)$ and by dominated convergence
for $x>0$ \eqref{subcrit_smol_frozen_integal_eq} is transformed
into
 \be \label{general_frozen_Laplace_V}
   V(t,x)=V(0,x)+\int_0^t V'(s,x)
    \left( -V(s,x)+ \left(m_0(0)-\Phi(s)\right) +\lambda(s)\right)
    ds
 \ee
In the sequel we denote the derivative of functions $f(t,x)$
 with respect to the time and space variables by $\dot f (t,x)$ and $f'(t,x)$, respectively.

 Let  \be \label{def_U} U(t,x):=V(t,x)-(m_0(0)- \Phi(t)) \ee
Thus \eqref{general_frozen_Laplace_V} is transformed into
 \be \label{inhom_subcrit_U}  U(t,x)= U(0,x)+\int_0^t
-U(s,x)\deri{U}{x}(s,x)+\lambda(s)\deri{U}{x}(s,x)ds +\Phi(t)
 \ee
 Since $V(t,\cdot)$ is a Laplace transform we have
\begin{equation}\label{U_properties1}
 U(t,0)=-\theta(t) \qquad \deri{U}{x}(t,0)=-m_1(t) \qquad
 \lim_{x \to \infty} U(x)=-m_0(0)+\Phi(t)
 \end{equation} and $U$ is a monotone
decreasing convex function of the variable $x$  for every $t$.

\begin{definition} \label{def_of_X}
Denote by $X(t,u)$ the inverse function of $U(t,x)$ with respect
to $x$, that is $U(t,X(t,u))=u$.

The domain of $X(t,u)$ in the variable $u$ is $(-m_0(t)+\Phi(t)
,-\theta(t) \rbrack $.
\end{definition}
\be \label{X_theta_zero} X(t,-\theta(t)) = 0 \ee
The notion of $X(t,\cdot)$ and a version of the following lemma
is already present in \cite{ZiffErnstHendriks}.
\begin{lemma}\label{lemma_X_solution} If $X(t,u)$ is defined using a solution of
 the general frozen percolation equation
 then the following
 identity holds:
 \be \label{X_solution_inhom_subcrit}
 X(t,u)=X\left(0,u-\Phi(t)\right) +t\cdot(u-\Phi(t))-\int_0^t \lambda(s)ds
 +
 \int_0^t \Phi(s)ds
 \ee
\end{lemma}
\begin{proof}  We fix an $x_{min}>0$. For any $x\geq x_{min}$ we have
 \be \label{trivibounds}
\abs{U(t,x)} \leq m_0(0), \quad \abs{\deri{U}{x}(t,x)} \leq
\frac{m_0(t)}{x_{min}}, \quad \abs{\kderi{U}{x}(t,x)}\leq
\frac{m_0(t)}{x_{min}^2},\ee
 moreover  $\sup_{0\leq t \leq T}
\lambda(t) < +\infty$. For an $x(0)> x_{min}$ denote by $x(t)$ the
solution of the integral equation \be \label{characteristic_curve}
x(t)=x(0) + \int_0^t U(s,x(s))-\lambda(s)ds \ee This equation is
well-posed on the domain $x(t)\geq x_{min}$, since
$U(s,x)-\lambda(s)$ is bounded and Lipschitz-continuous in $x$.

Moreover \[x(t+dt)-x(t)=\Ordo(dt),\quad
\abs{U(t,x(t))-U(t,x(t+dt))}=\Ordo(\frac{dt}{x_{min}}).\] If we
differentiate \eqref{inhom_subcrit_U} w.r.t. $x$ we get
$\abs{\deri{U}{x}(t+dt,x)-\deri{U}{x}(t,x)}
=\Ordo(\frac{dt}{x_{min}^2})$.
 \begin{multline*} U(t+dt,x(t+dt))-U(t,x(t)) - \left(\Phi(t+dt)-\Phi(t)\right)=\\
 \big(U(t+dt,x(t+dt))-U(t,x(t+dt))\big)+\\ \big(U(t,x(t+dt))-U(t,x(t))\big)-\left(\Phi(t+dt)-\Phi(t)\right) =\\
 \int_t^{t+dt}-U(s,x(t+dt))\deri{U}{x}(s,x(t+dt)) +\lambda(s) \deri
{U}{x}(s,x(t+dt))ds + \\
 \deri{U}{x}(t,x(t+dt))\int_t^{t+dt} U(s,x(s))-\lambda(s)ds
 +\Ordo(\frac{dt^2}{x_{min}^2})=\\
 \int_t^{t+dt} U(s,x(t+dt)) \big( \deri{U}{x}(t,x(t+dt))- \deri{U}{x}(s,x(t+dt)) \big) ds +\\
  \int_t^{t+dt} \deri{U}{x}(t, x(t+dt)) \big(
 U(s,x(s))-U(s,x(t+dt)) \big) ds + \\ \int_t^{t+dt}
 \lambda(s)\big(\deri{U}{x}(s,x(t+dt))-\deri{U}{x}(t,x(t+dt)) \big) ds +
 \Ordo(\frac{dt^2}{x_{min}^2})= \Ordo(\frac{dt^2}{x_{min}^2})
\end{multline*}
Thus $U(t,x(t))=U(0,x(0))+\Phi(t)$, and if we substitute this back
into \eqref{characteristic_curve}, we get
\[x(t)=x(0)+t U(0,x(0))+\int_0^t \Phi(s)ds -\int_0^t \lambda(s)ds
\]
By the definition of $X(t,u)$ we have $X(t,U(t,x(t)))=x(t)$, and
by substituting \[u=U(0,x(0))+\Phi(t)\] we obtain
\eqref{X_solution_inhom_subcrit}.
\end{proof}

Since $\conf{v}(0) \in \myNz$, $V(0,x)$ is well-defined and
analytic for all $x \in \R$, thus $X(0,u)$ can be analytically
extended to $(-m_0(0), +\infty)$. \eqref{X_solution_inhom_subcrit}
makes it possible  to extend $X(t,u)$ to $(-m_0(0)+\Phi(t),
+\infty)$ analytically. The extended $X(t,u)$ is a strictly convex
function of the $u$ variable. If we differentiate
\eqref{X_solution_inhom_subcrit} w.r.t. $u$, we get \be
\label{X_derivative_solution} \deri{X}{u}(t,u)=\deri{X}{u}(0,
u-\Phi(t))+t\ee

\begin{definition} \label{def_G_F_E} Define $F(t,w)$ by the identity
\be \label{defF} F(t,-\deri{X}{u}(t,u))=-u \ee Thus $-F(t,w)$ is
the inverse function of $-X'(t,u)$. If $\hat{X}$ denotes the
Legendre-transform of $X$ w.r.t. the variable $u$, then \be
\label{defG} G(t,w):=\hat{X}(t,-w)=-\min_{u} \{ wu+X(t,u) \}= w
F(t,w)-X(t,-F(t,w)) \ee Let
\begin{equation}\label{def_critical_core}
E(t,w)=\kderi{G}{w}(t,w)=\deri{F}{w}(t,w).
\end{equation} We call
$E(t,\cdot)$ the \emph{critical core} of $\conf{v}(t)$. If we use the
extended definition of $X$ then $G(t,w)$ is well-defined and
analytic for all $w>-t$.
\end{definition}
We have
\be \label{F_U_E_U_formulas}
F(t,-\frac{1}{U'(t,x)})=-U(t,x)\quad \text{and}\quad E(t, -\frac{1}{U'(t,x)})=
\frac{(-U'(t,x))^3}{U''(t,x)}
\ee

It follows from the properties of the Legendre-transformation and
\eqref{X_solution_inhom_subcrit} that \be
\label{G_solution_inhom_subcrit} G(t,w)=G(0,w+t)-w\cdot
\Phi(t)-\int_0^t\Phi(s)ds +\int_0^t \lambda(s)ds
 \ee
\be \label{F_solution_inhom_subcrit}
F(t,w)=F(0,w+t)-\Phi(t)
 \ee
\be \label{E_solution_translated}
E(t,w)=E(0,w+t)
\ee

$G(t,\cdot)$ is strictly convex and $G$ determines $X$ uniquely
since the Legendre-transformation is invertible. Define \be
\label{def_of_w_star} w^*(t):= -\deri{X}{u}(t,0)\; \iff\;
F(t,w^*(t))=0 \; \iff \; \text{argmin}_w G(t,w)= w^*(t)
 \ee
 \be
\label{G_F_w_star_using_X_zero}
 X(t,0)=0 \; \implies \;
 G(t,w^*(t))=0\;
 \implies\;
 \forall w \; G(t,w) \geq 0 \ee
 \be \label{theta_zero_implies_w_star_positive} \theta(t)=0
\quad \implies \quad w^*(t)=\frac{1}{m_1(t)} \geq 0
 \ee
 \be
\label{x_star_and_X_G} x^*(t)=\inf \lbrace x: \sum_{k=1}^{\infty}
v_k(t) e^{-kx} < +\infty \rbrace =
 \min_{u} X(t,u)= X(t,-F(t,0))=-G(t,0) \ee

\section{The frozen percolation equations are
well-posed}\label{section_wellposed}

\begin{lemma}\label{alter_eq_well_posed}
The alternating equation \eqref{alter_smol_frozen_eq} is
well-posed.
\end{lemma}
\begin{proof}
If we are given the sequence of burning times $0 < \bt_1 < \bt_2<
\dots$ the solution of \eqref{alter_smol_frozen_eq} can be
uniquely constructed by using induction on $i$: if we already have
the solution on $\lbrack 0, \bt_i \rbrack$, then we are given
$m_0(\bt_i)$, so we can uniquely solve the sequence of ordinary
differential equations \eqref{alter_smol_frozen_eq} for $v_1, v_2,
\dots$ on $\lbrack \bt_i, \bt_{i+1} \rbrack$ by repeatedly
applying the Picard-Lindel\"of theorem, since the equation for $v_k$
only contains $v_1,\dots, v_k$ on its right-hand side.
\end{proof}

\begin{lemma} \label{subcrit_integral_eq_uniqueness}
The solution of the integral equations
\eqref{subcrit_smol_frozen_integal_eq} is unique for every initial
condition $\conf{v}(0) \in \myNz$ if $\lambda(t)$ is nonnegative
and continuous.
\end{lemma}
\begin{remark}
Choosing $\lambda(t)\equiv 0$ implies the uniqueness of the
solutions of \eqref{crit_smol_frozen_integral_eq}.
\end{remark}
\begin{proof}
In order to prove the uniqueness of the solution of
\eqref{subcrit_smol_frozen_integal_eq}, we only have to prove that
given two solutions with the same initial condition, the function
$\Phi(t)=m_0(0)-m_0(t)$ determined by the two solutions is the
same, because $m_0(t)$ and \eqref{subcrit_smol_frozen_integal_eq}
determines $v_k(t)$ for all $k$ uniquely. For a solution
$\conf{v}(t)$ of \eqref{subcrit_smol_frozen_integal_eq} we can
define $U$ by \eqref{def_U}, then $X$ by Definition
\ref{def_of_X}., which satisfies \eqref{X_solution_inhom_subcrit}
and the $G$ of Definition \ref{def_G_F_E}. satisfies
\eqref{G_solution_inhom_subcrit}.

 Assume that $G_1$ and
$G_2$ are obtained this way from two solutions of
\eqref{subcrit_smol_frozen_integal_eq} with the same initial
condition $G(0,w)$. Let $\tilde{G}=G_1-G_2$ and
$\tilde{\Phi}=\Phi_1-\Phi_2$. Then
\[  \tilde{G}(t,w)=-w\cdot
\tilde{\Phi}(t)-\int_0^t\tilde{\Phi}(s)ds \] Now by
\eqref{subcrit_Phi_def} we have $\theta(t)=0$, thus
\eqref{X_theta_zero} $\implies \, X(t,0)=0$, and
\eqref{G_F_w_star_using_X_zero} $\implies \,
 \min_w
G_1(t,w)=\min_w G_2(t,w)=0$ and
\eqref{theta_zero_implies_w_star_positive} $\implies \,
w^{*}_i(t):=\text{argmin}_w G_i(t,w) \geq 0$ for $i=1,2$, thus we
have $\tilde{G}(t,w^{*}_1(t))\leq 0$ and
$\tilde{G}(t,w^*_2(t))\geq 0$. Thus $\tilde{\Phi}(t)$ and
$\int_0^t\tilde{\Phi}(s)ds$ cannot have the same sign. But if
$(t_1, t_2)$ is a maximal interval such that for $t_1<t<t_2$ we
have $\int_0^t\tilde{\Phi}(s)ds>0$ then
$\int_0^{t_1}\tilde{\Phi}(s)ds=0$ and \[t \in \lbrack t_1, t_2
\rbrack \implies \int_0^t\tilde{\Phi}(s)ds \geq 0 \implies
\tilde{\Phi}(t) \leq 0 \implies \int_{t_1}^t \tilde{\Phi}(s)ds
\leq 0
\]
which contradicts the definition of $t_1$ and $t_2$. Thus
$\int_0^t\tilde{\Phi}(s)ds \leq 0$ for all $t$ and interchanging
the roles of $G_1$ and $G_2$ we get $\int_0^t\tilde{\Phi}(s)ds
\equiv 0$, so $\Phi_1(t) \equiv \Phi_2(t)$.
\end{proof}

\begin{lemma} \label{existence_lemma_1}
 If we find a function $\varphi(t)$ such that defining
$\Phi(t):=\int_0^t \varphi(s)ds$ and $G(t,w)$ by
\eqref{G_solution_inhom_subcrit} we have
 \be\label{subcrit_G_condition} \min_{w} G(t,w) = 0 \quad   \text{ and } \quad
w^*(t)=\mathrm{argmin}_{w} G(t,w)\geq 0 \ee for all $t$, then the
solution of \eqref{general_frozen_eqs} with the same
$\lambda(\cdot)$, $\Phi(\cdot)$ and initial condition satisfies
\eqref{subcrit_smol_frozen_integal_eq}.
\end{lemma}
\begin{proof}
Since the Legendre-transformation is invertible, from
 \eqref{subcrit_G_condition} we get
\[X(t,0)=0 \quad \text{ and } \quad
X'(t,0)\leq 0.\] $X(t,u)$ is strictly decreasing for $u<0$, thus
it is the inverse function of an $U(t,x)$ satisfying $U(t,0)=0$.
If we plug $\Phi(\cdot)$ into \eqref{general_frozen_eqs} then we
get $\theta(t)=-U(t,0)=0$, therefore \eqref{subcrit_Phi_def} is
satisfied.
\end{proof}

\begin{lemma}\label{crit_eq_existence}
The $\Phi$ of the unique solution of
\eqref{crit_smol_frozen_integral_eq} is \be \Phi(T) = \left\{
\begin{array}{ll}
0 & \mbox{if $t \leq \gt$}\\
F(0,T) & \mbox{if $t \geq \gt$}
\end{array}
\right. \ee
where $\gt=\frac{1}{m_1(0)}$.
\be\label{kritikus_celfuggveny} \int_0^T \Phi(t)dt = \left\{
\begin{array}{ll}
0 & \mbox{if $T \leq \gt$}\\
G(0,T) & \mbox{if $T \geq \gt$}
\end{array}
\right. \ee

\end{lemma}
\begin{proof}
The solution is unique according to Lemma
\ref{subcrit_integral_eq_uniqueness}. and to prove its existence
we only have to find a function $\varphi(t)$ that satisfies the
criteria of Lemma \ref{existence_lemma_1} (with $\lambda(t) \equiv
0$). We will show that \be \label{crit_varphi_def_formula}
\varphi(t)=\ind \lbrack t \geq \frac{1}{m_1(0)} \rbrack E(0,t)\ee
 does the job. For $t \leq \gt$
this is trivial by looking at \eqref{G_solution_inhom_subcrit}:
$G(t,w^*(t))=0$ and $w^*(t)=\frac{1}{m_1(0)}-t \geq 0$ if
$\Phi(t)\equiv 0$.

 We will show that
for $t \geq \gt$ we have $G(t,0) \equiv 0$ and $F(t,0) \equiv 0$,
that is $w^*(t) \equiv 0$. $F(0,\gt)=G(0,\gt)=0$ by
\eqref{G_F_w_star_using_X_zero} and $w^*(0)=\frac{1}{m_1(0)}=\gt$.
  $F(t,0)=0$ follows from \eqref{F_solution_inhom_subcrit} and
\[\Phi(t)=\int_0^t \varphi(s)ds=\int_{\gt}^t E(0,s)ds=F(0,t)-F(0,\gt)=F(0,t)\]
By \eqref{G_solution_inhom_subcrit} we have
\[ G(t,0)= G(0,t)-\int_0^t \Phi(s)ds = \int_{\gt}^t F(0,s)ds- \int_{\gt}^t F(0,s)ds=0 \]

\end{proof}

The well-posedness of the integral equation
\eqref{crit_smol_frozen_integral_eq} implies that of the
corresponding differential
 equation,
since $m_0(0)-\Phi(t)=m_0(t)$ is a continuous function of $t$, thus $v_k(t)$ are
 differentiable.

We have shown that the solution of
\eqref{crit_smol_frozen_integral_eq} has infinite first moment
after the gelation time:
 $\frac{1}{w^*(t)}=m_1(t)=+\infty$ for all $t \geq \gt$.

\begin{definition}\label{phiinf_phisup_def}
Let $E(0,w)$ denote the critical core of $\conf{v}(0)$ (see
Definition \ref{def_G_F_E}).

 For $\frac{1}{m_1(0)} \leq w_1\leq w_2$
define \[\phiinf(w_1,w_2):=\min_{w_1 \leq w \leq w_2}E(0,w) \quad
\text{ and } \quad \phisup(w_1,w_2):=\max_{w_1 \leq w \leq
w_2}E(0,w).\]
\[\phisup:=\phisup(\frac{1}{m_1(0)},+\infty), \quad
\phiinf(w):=\phiinf(\frac{1}{m_1(0)},w)\]
\end{definition}

\begin{lemma}\label{lemma_upper_lower_bound_on_E}
If $w \geq \frac{1}{m_1(0)}$ then the inequalities \be
\label{upper_lower_bound_on_E} \frac{m_1(0)}{m_2(0)}\frac{1}{w^2}
\leq E(0,w) \leq \frac{1}{w^2} \ee hold. Thus $\phisup\leq
m_1(0)^2$ and $\phiinf(w)\geq\frac{m_1(0)}{m_2(0)}\frac{1}{w^2}$.

 For all $w \geq
\frac{1}{m_1(0)}$ we have \be \label{Lipschitz}
\abs{\deri{E}{w}(0,w)} \leq 4 m_2(0)^2 m_3(0) =:D \ee which
implies \be \label{lipschitz_phisup_phiinf}
\phisup(w_1,w_2)-\phiinf(w_1,w_2) \leq D\cdot(w_2-w_1) \ee

\end{lemma}
\begin{remark}\label{monodisperse_E_remark}
 If $m_1(0)=m_2(0)$ then the
upper and lower bounds in \eqref{upper_lower_bound_on_E} coincide.
This can only happen if $v_k(0)=m_1(0) \cdot \ind\lbrack k=1
\rbrack$, this is the case known as the monodisperse initial
condition (the initial graph has no edges).
\end{remark}
\begin{proof}
Let $U(x):=U(0,x)$. Recalling \eqref{F_U_E_U_formulas}  $E
\left(0, -\frac{1}{ U'(x)}\right)=\frac{(-U'(x))^3}{ U''(x)}$
holds. The upper bound of \eqref{upper_lower_bound_on_E} follows
from $-U'(x) \leq U''(x)$, and $-U'(x) \frac{m_2(0)}{m_1(0)} \geq
U''(x)$ holds because $\log \left( -U'(x) \right)$ is a convex
function, thus $\frac{U''(x)}{U'(x)} \geq
\frac{U''(0)}{U'(0)}=\frac{m_2(0)}{-m_1(0)}$. The bound on the
Lipschitz constant \eqref{Lipschitz} follows from
\[\abs{\deri{E}{w}\left(0, -\frac{1}{
U'}\right)}=\abs{ \frac{(U')^5
U'''}{(U'')^3}-3\frac{(U')^4}{U''}}\leq \abs{(U')^2
U'''}+3\abs{(U')^3} \leq 4m_2(0)^2 m_3(0) \]
\end{proof}

Now we turn our attention to the subcritical equation
\eqref{subcrit_smol_frozen_integal_eq}. We assume $\lambda(t)>0$
for all $t$. If we substitute $x=0$ into the differential equation
(\ref{inhom_subcrit_U}) and assume
$\abs{\deri{U}{x}(t,0)}<+\infty$ then (formally) we get
\[\dot{\Phi}(t)= \varphi(t)=-\deri{U}{x}(t,0) \cdot
\lambda(t)=m_1(t)\lambda(t)=\frac{\lambda(t)}{w^*(t)},\]

% If we take the time derivative of
%(\ref{X_derivative_solution}), we get
%\[\dot{w}^*(t)=\kderi{X}{u}(t,0)\varphi(t)-1=\frac{\varphi(t)}{E(t,w^*(t))}-1=
%\frac{\lambda(t)}{w^*(t)E(0,t+w^*(t))}-1 \]

\begin{definition} If $\conf{v}(0) \in \myNz$ and
$\lambda(t)$ is a positive continuous function then the
subcritical control differential equation for $w^*(t)$ is \be
\label{subcritical_control_ODE} \dot{w}^*(t)=
\frac{\lambda(t)}{w^*(t)E(0,t+w^*(t))}-1 \ee with initial
condition $w^*(0)=\frac{1}{m_1(0)}=\gt$.
\end{definition}
\begin{lemma}\label{subcritical_existence_lemma}
The subcritical control differential equation is well-posed and
the function \[\varphi(t):=\frac{\lambda(t)}{w^*(t)}\]
(where $w^*(t)$ is the solution of \eqref{subcritical_control_ODE} with  $w^*(0)=\frac{1}{m_1(0)}$)
 satisfies the
criteria of Lemma \ref{existence_lemma_1}, which implies the
existence of solutions to \eqref{subcrit_smol_frozen_integal_eq}.
\end{lemma}
\begin{proof}
We prove  the statement of the lemma on $\lbrack 0, T \rbrack$.
The Picard-Lindel\"of theorem and the Lipschitz-continouity property
 (\ref{Lipschitz}) gurantee the existence and uniqueness
of the solution of (\ref{subcritical_control_ODE}) before the
graph of the solution exits \be \label{box}\{ (t,w^*): 0 \leq t
\leq T,\quad w^*_{min} \leq w^* \leq w^*_{max},\quad w^*+t \geq
w^*(0) \} \ee for some $0<w^*_{min}<w^*_{max}<+\infty$.

 Let
$\lambda_{inf}:=\inf_{0\leq t\leq T} \lambda(t)$,
$\lambda_{sup}:=\sup_{0\leq t\leq T} \lambda(t)$. From
(\ref{subcritical_control_ODE}) and a ``forbidden region''-type argument
 we get
that $w^*(t)+t \geq w^*(0)$ and $w^*(t) \geq \min\{
\frac{\lambda_{inf}}{\phisup}, w^*(0) \}$, since \[w^*(t)> 0
\implies \frac{d}{dt}( w^*(t)+t) \geq 0\]\[
 w^*(t)+t \geq w^*(0)
\implies E(0,t+w^*(t)) \leq \phisup,\] thus $w^*(t) <
\frac{\lambda_{inf}}{\phisup} \implies \dot{w}^*(t) >0$.

Now we
prove that $w^*(t)$ cannot grow too fast
 using the lower bound of (\ref{upper_lower_bound_on_E}). $w^*(t) \leq y(t)$ where $y(0)=w^*(0)=\gt$ and
\[\dot{y}(t)=\lambda_{sup}\frac{m_2(0)}{m_1(0)}\frac{(y(t)+t)^2}{y(t)}\leq
\lambda_{sup}\frac{m_2(0)}{m_1(0)}\left(\frac{\gt+t}{\gt}\right)\cdot (y(t)+t) \]
since $y(t)$ is increasing.
 Thus $\dot{y}(t) \leq a\cdot  y(t)+b$ for some $a$ and $b$
 depending only on the initial data, the function $\lambda(t)$ and
 $T$. Thus
 \[w^*(t)\leq w^*(0) e^{a t} +\frac{b}{a} \cdot (e^{at} -1).\]

 Now we can see that the graph of the solution of
(\ref{subcritical_control_ODE}) indeed doesn't exit (\ref{box})
until $t=T$  if we define
 \be \label{w_min_w_max}
   w^*_{min}= \min\{ \frac{\lambda_{inf}}{\phisup},
\gt \} \quad \text{ and }\quad
w^*_{max}:=(\gt+\frac{b}{a})e^{aT}\ee Now we prove that
$\varphi(t):=\frac{\lambda(t)}{w^*(t)}$ satisfies the criteria of
Lemma \ref{existence_lemma_1}. by showing that  \[G(t,w^*(t))\equiv
0 \quad \text{ and } \quad  F(t,w^*(t)) \equiv 0.\] This holds for $t=0$, so it
suffices to check $\frac{d}{dt} G(t,w^*(t))\equiv 0$ and
$\frac{d}{dt} F(t,w^*(t))\equiv 0$. Using
(\ref{F_solution_inhom_subcrit})
\begin{multline*}
\frac{d}{dt} F(t,w^*(t))= E(0,t+w^*(t))\cdot \left(
1+\frac{\lambda(t)}{w^*(t) E(0,t+w^*(t))} -1\right)
-\frac{\lambda(t)}{w^*(t)} =0
\end{multline*} If we combine $F(t,w^*(t)) \equiv 0$ with (\ref{F_solution_inhom_subcrit}) we get \be \label{Phi_and_F}
F(0,t+w^*(t))=\Phi(t) \ee It is straightforward to verify
$\frac{d}{dt} G(t,w^*(t))\equiv 0$ by using
(\ref{G_solution_inhom_subcrit}) and (\ref{Phi_and_F}).
\end{proof}
This completes the proof of the well-posedness of
(\ref{subcrit_smol_frozen_integal_eq}).

\section{Proof of Theorem
 \ref{thm_convergence_of_process_to_eqn}.}
\label{proof_of_convergence}
We consider the sequence $\Prb_N$ of probability measures on the compact
 space $\rbspace_{\conf{w}} \lbrack 0, T \rbrack$.
From Prokhorov's theorem it follows that any subsequence of the
measures $\Prb_N$ contains a sub-subsequence that converges weakly
to a limiting measure on $\rbspace_{\conf{w}} \lbrack 0, T
\rbrack$.

\begin{lemma}\label{lemma_weaklimit}
Any weak limit point of the measures $\Prb_N$ is concentrated on the
set of solutions of the general frozen percolation equation
\eqref{general_frozen_eqs}.
\begin{itemize}
\item If $\mu(N)\equiv 1$, then the $\lambda(t)$ rate function of
\eqref{general_frozen_eqs} is equal to the $\lambda(t)$ of
\eqref{lightning_rate}.
\item If $\mu(N) \ll 1$, then the $\lambda(t)$ rate function of
\eqref{general_frozen_eqs} is equal to $0$.
\end{itemize}

\end{lemma}
\begin{proof}
From  \eqref{coag_rate} and \eqref{lightning_rate}  it follows that
\begin{multline} \label{frozen_infinitesimal_expected_change_v}
L v_k^N(t):=\lim_{dt \to 0}
\condexpect{v_k^N(t+dt)-v_k^N(t)}{\cF_t}=\\
 \frac{1}{N}
\frac{\cV_k(\cV_k-k)}{2}\left(-2 \frac{k}{N}\right)+
\left(\sum_{l \neq
k} \frac{1}{N} \cdot \cV_k \cV_l  \right)\left(-\frac{k}{N}\right) +\\
\left(\sum_{l=1}^{\lfloor \frac{k-1}{2}\rfloor} \frac{1}{N}
\cV_l \cV_{k-l}
 + \ind \lbrack 2|k \rbrack
\frac{1}{N} \frac{(\cV_{\frac{k}{2}}-\frac{k}{2})
\cV_{\frac{k}{2}}}{2} \right)\frac{k}{N}
-\lambda(t)\cdot \mu(N) \cV_k \frac{k}{N}=\\
-k \cdot((m_0(0)-\Phi^N(t))+\lambda(t)\mu(N)) \cdot \rv_k^N
+\frac{k}{2} \sum_{l=1}^{k-1}\rv_l^N \rv_{k-l}^N
+\frac{1}{N}\left(k^2 \rv_k^N - \ind \lbrack 2 | k \rbrack \cdot
\frac{k^2}{4} \rv_{\frac{k}{2}}^N \right)
\end{multline}
$M(t)=v_k^N(t)-v_k^N(0) -\int_0^t L v_k^N(s)ds$ is a martingale and
\begin{multline*}
L M^2(t):= \lim_{dt \to 0} \condexpect{M^2(t+dt)-M^2(t)}{\cF_t}=
\lim_{dt \to 0} \condexpect{(v_k^N(t+dt)-v_k^N(t))^2}{\cF_t}\leq \\
\left(2\frac{k}{N}\right)^2 \cdot \left( \binom{ \lfloor m_0(0) N \rfloor }{2}
 \frac{1}{N}
+ \lfloor m_0(0) N \rfloor \lambda(N) \right) =\Ordo\left( \frac{k^2}{N} \right)
\end{multline*}
Thus $\expect{M(T)^2}=\expect{\int_0^t L M^2(s)ds}=
 \Ordo\left(\frac{1}{N}\right)$ if we fix $k$. It follows
from Doob's maximal inequality that for all $\varepsilon>0$,
$k\geq 1$ and $T<+\infty$ we have \be
\label{convergence_in_prob_to_general_frozen} \lim_{N \to \infty}
\prob{ \sup_{0\leq t \leq T} \left| v_k^N(t) -v_k^N(0) - \int_0^t
L v_k^N(s)ds\right| > \varepsilon}=0 \ee If we rewrite this
equation in terms of the functions
$\left(w_k^N(\cdot)\right)_{k=1}^{\infty}$ the claim of the lemma
follows.
\end{proof}

\begin{lemma}
\label{no_giant_lemma} If $\frac{1}{N} \ll \mu(N)$,
$0<\lambda_{inf} \leq \lambda(t)$ and $\conf{v}(0) \in \myNz$,
 then for any weak limit point $\Prb$  of the sequence
of probability measures $\Prb_{N}$ on
 $\rbspace_{\conf{w}} \lbrack 0,
T \rbrack$ we have
\begin{equation}
\label{no_giant_formula} \Prb \left(\theta(t) \equiv 0 \right)=1
\end{equation}
\end{lemma}

The subcritical and critical
 parts of
Theorem \ref{thm_convergence_of_process_to_eqn}. follow from Lemma
\ref{lemma_weaklimit}. and Lemma \ref{no_giant_lemma}.:  any weak
limit point $\Prb$ of the sequence $\Prb_N$ is concentrated on the set
of frozen percolation evolutions satisfying
\eqref{general_frozen_eqs} \& \eqref{subcrit_Phi_def}. When
$\mu(N) \equiv 1$, $\Prb$ is concentrated on the unique solution of
\eqref{subcrit_smol_frozen_integal_eq}, when $\frac{1}{N} \ll
\mu(N) \ll 1 $ then $\Prb$ is concentrated on the solution of
\eqref{crit_smol_frozen_integral_eq}.

In the rest of this section we discuss the proof of Lemma \ref{no_giant_lemma}.

\begin{lemma}
We consider a solution of the general frozen percolation equation
\eqref{general_frozen_eqs} with initial condition $\conf{v}(0) \in
\myNz$. If $\, \lambda(t)\equiv 0$ or $0<\lambda_{inf} \leq
\lambda(t) \leq \lambda_{sup} <+\infty$ then there is a constant
$C^*$  such that for all $ t_1 \leq t_2 $ we have
\begin{equation}
\label{giant_lipchitz} \theta(t_2)-\theta(t_1) \leq C^*
\cdot(t_2-t_1)
\end{equation}
\end{lemma}
\begin{proof}
First we prove that there exists a constant $C$ depending only on
the initial data $\conf{v}(0)$ and $\lambda_{inf}$ such that \be
\label{m_1_bound_for_general_frozen_eq}
 m_1(t)\leq C \ee

 If $V(t,x)=\sum_{k=1}^{\infty} v_k(t)e^{-kx}$ then by \eqref{general_frozen_eqs}
 we get
\begin{align}
\label{general_frozen_eq_V}
 \dot{V}(t,x)&=& V'(t,x) \cdot \left(
(m_0(0)-\Phi(t))+\lambda(t)-V(t,x)
\right)\\
\label{general_frozen_eq_derivative_x}
\dot{V}'(t,x)&=&V''(t,x)\left(m_0(0)-\Phi(t)-\lambda(t)-V(t,x)\right)-V'(t,x)^2
\end{align}
Substituting $V(t,x)-(m_0(0)-\Phi(t)) \leq 0$ and $\frac{
-V'(t,x)^3}{\phisup} \leq V''(t,x)$ into
\eqref{general_frozen_eq_derivative_x} we get
\[ \frac{d}{dt} \left(-V'(t,x)\right)\leq V'(t,x)^2 \cdot
\left(1-\frac{\lambda_{inf}}{\phisup} \left(-V'(t,x)\right)
\right)
\]
which implies $-V'(t,x)\leq \max\{
m_1(0),\frac{\phisup}{\lambda_{inf}} \}=:C$ for all $x>0$ and $t$
by a "forbidden region"-argument. Thus by letting $x \to 0_+$ we get
\eqref{m_1_bound_for_general_frozen_eq}.

Now we show that for some constant $C_2$ we have \be
\label{bound_on_U_V_prime}
 \left(
V(t,x)-(m_0(0)-\Phi(t) \right)V'(t,x) \leq C_2
 \ee
for all $x>0$. If $\lambda_{inf} \leq \lambda(t)$, then by
\eqref{m_1_bound_for_general_frozen_eq} and $-m_0(0) \leq
V(t,x)-(m_0(0)-\Phi(t)) \leq 0$ we get \eqref{bound_on_U_V_prime}
with $C_2=m_0(0)C$.

Denote by $U(t,x):=V(t,x)-(m_0(0)-\Phi(t))$.
 If $\lambda(t)\equiv 0$ then by
 \eqref{general_frozen_eq_V} and
\eqref{general_frozen_eq_derivative_x} we get
\begin{multline*}
 \frac{d}{dt}
\left( U(t,x) V'(t,x)\right)= -2 V'(t,x)^2 U(t,x)-U(t,x)^2
V''(t,x) + V'(t,x) \frac{d}{dt}\Phi(t) \leq \\
(-U(t,x))V'(t,x)^2 \left(2- \frac{1}{\phisup} U(t,x) V'(t,x)
\right)
\end{multline*}
Thus we have  \eqref{bound_on_U_V_prime} with $C_2=\max \{ m_1(0),
2\phisup \}$ again by a "forbidden region"-argument. Substituting the bounds
\eqref{m_1_bound_for_general_frozen_eq} and
\eqref{bound_on_U_V_prime} into \eqref{general_frozen_eq_V} we get
\[ \frac{d}{dt} \left( -V(t,x) \right) \leq C_2+C\cdot \lambda_{sup}=: C^* \]
for all $x$. Thus $V(t_1,x)-V(t_2,x) \leq C^* \cdot (t_2-t_1)$.
Letting $x \to 0_+$ and substituting into
\eqref{theta_def_nonnegative_general_frozen_eq} the claim of the
lemma follows.
\end{proof}

We are going to prove Lemma \ref{no_giant_lemma} by
contradiction: in Lemma \ref{lemma:disjoint_intervals_big_mass} we show that if $\theta(\cdot)\not \equiv 0$ in the limit, then there is a positive time interval  such that $\theta(t)$ has a positive lower bound, and that this implies that even in the convergent sequence of finite-volume models, a lot of mass is contained in arbitrarily big components on this interval. Than in subsequent Lemmas  we prove that these big components indeed burn, which produces such a big increase in the value of the burnt mass $\Phi(\cdot)$ that is in contradiction with
$\Phi(\cdot) \leq m_0(0)$.

For any frozen percolation evolution obtained from a frozen
percolation Markov process on a finite number of vertices we
obviously have $\theta^N(t)\equiv 0$ (see
\eqref{frozen_evolution_for_N} and
\eqref{theta_def_nonnegative_general_frozen_eq}), thus
\be\label{no_giant_for_finite_N}
 \forall
K\in \N \quad \sum_{k > K} v_k^N(t)=
m_0(t)-w_K^N(t)=m_0(0)-\Phi^N(t)-w_K^N(t) \ee

\begin{lemma}
\label{lemma:disjoint_intervals_big_mass} If $ \Prb_N \weak \Prb$
where $\Prb$ does not satisfy \eqref{no_giant_formula} on $\lbrack
0, T \rbrack$, then there exist $\rbeps_1$, $\rbeps_2$,
$\rbeps_3>0$ and a deterministic $\spect\in[\rbeps_1,T]$ such that
for every $K<+\infty$, every $m<+\infty$ and every sequence
\[
\spect -\rbeps_1 <\alpha_1<\beta_1<\alpha_2<\beta_2
<\dots<\alpha_m<\beta_m < \spect
\]
there exists an $N_0<+\infty$ such that for every $N\geq N_0$ and
$1 \leq i \leq m$ we have
\begin{equation}
\label{disjoint_intervals_big_mass} \Prb_N \left(\max_{\alpha_i \leq
t \leq \beta_i} \sum_{k>K}v_k^N(t)
>\rbeps_2 \right)
> \rbeps_3.
\end{equation}
\end{lemma}

\begin{proof}
First we prove that if $\Prb$ does not satisfy
\eqref{no_giant_formula} then there exist
$\rbeps_1,\rbeps_2,\rbeps_3>0$ and $\rbeps_1 \leq \spect \leq T$
such that
\begin{equation}
\label{det_time_ineterval_giant} \Prb \big(\inf_{\spect-\rbeps_1
\leq t \leq \spect}\theta(t) > \rbeps_2\big) > \rbeps_3.
\end{equation}
Since \eqref{no_giant_formula} is violated, we have $ \Prb\big(
\sup_{0\leq t \leq T} \theta(t) > \rbeps\big)>\rbeps $ for some
$\rbeps>0$.

Let $L:=\lfloor \frac{2C^* T}{\rbeps} \rfloor$ and
$t_i:=\frac{\rbeps i}{2C^*} $ for $1 \leq i \leq L$ where $C^*$ is
the constant in \eqref{giant_lipchitz}.

 By Lemma \ref{lemma_weaklimit}. the random frozen percolation
evolution obtained as a weak limit point satisfies
\eqref{general_frozen_eqs} with a possibly random control function
$\Phi$, so \eqref{giant_lipchitz} holds $\Prb$-almost surely for the
random element of $\rbspace_{\conf{w}} \lbrack 0, T \rbrack$
obtained as a weak limit point.

 Since $\theta(0)=0$ we have
\[
\big\{ \sup_{0\leq t\leq T} \theta(t) > \rbeps \big\} \subseteq
\bigcup_{i=1}^L \big\{\theta(t_i)> \frac{\rbeps}{2} \big\}
\]
almost surely with respect to $\Prb$. Thus $\Prb \big(\theta(\spect)>
\frac{\rbeps}{2}\big) > \frac{\rbeps}{L}$ for some $\spect \in \{
t_1,\dots t_L \}$. Using \eqref{giant_lipchitz} again
\eqref{det_time_ineterval_giant} follows with
$\rbeps_1:=\frac{\rbeps}{4C^*}$, $\rbeps_2:=\frac{\rbeps}{4}$,
$\rbeps_3=\frac{\rbeps}{L}$.

Now given $K$ and the intervals $\lbrack \alpha_i,\beta_i
\rbrack$, $1 \leq i \leq m$ we define the continuous functionals
$f_i: \rbspace_{\conf{w}} \lbrack 0, T \rbrack \to \R$ by
\[
f_i \left( \left(w_k(\cdot) \right)_{k=1}^{\infty}, \Phi(\cdot)
\right):= \frac{1}{\beta_i-\alpha_i} \int_{\alpha_i}^{\beta_i}
\big(m_0(0)-w_K(t)-\Phi(t) \big)dt
\]
 Thus for all $i$
\[
H_i:= \{ \left( \left(w_k(\cdot) \right)_{k=1}^{\infty},
\Phi(\cdot) \right) \in \rbspace_{\conf{w}} \lbrack 0, T \rbrack
: f_i\left( \left(w_k(\cdot) \right)_{k=1}^{\infty}, \Phi(\cdot)
\right)>\rbeps_2 \}
\]
is an open subset of $\rbspace_{\conf{w}} \lbrack 0, T \rbrack$
with respect to the topology of Definition
\ref{def_frozen_percolation_evolution}. Thus by the definition of
weak convergence of probability measures we have
\[
\lim_{N \to \infty} \Prb_N(H_i) \geq \Prb(H_i) \geq
\Prb\left(\inf_{\spect-\rbeps_1 \leq t \leq \spect}\theta(t) >
\rbeps_2 \right) > \rbeps_3
\]
from which the claim of the lemma easily follows by
\eqref{no_giant_for_finite_N}.
\end{proof}

\begin{lemma}
\label{lemma:if_big_mass_big_fire_soon} If $\frac{1}{N} \ll
\mu(N)$ and $0<\lambda_{inf} \leq \lambda(t)$, then for every
$\rbeps_2>0$ there is a $\rbeps_4>0$ such that for every $\rbt>0$
there is a $K$ and an $N_1$ such that
\begin{equation}
\label{if_big_mass_big_fire_soon_eq} \forall N \geq N_1 \quad
 \sum_{k > K} v_k^N(0)  \geq
\rbeps_2 \; \implies \; \E_N \left(\Phi^N(\rbt) \right) \geq
\rbeps_4
\end{equation}
\end{lemma}
The proof of Lemma \ref{lemma:if_big_mass_big_fire_soon}. will
follow as a consequence of the Lemmas
\ref{lemma:almost_giant_or_lot_burnt}. and
\ref{lemma:if_almost_giant_then_big_fire_soon}.

\begin{proof}[Proof of Lemma \ref{no_giant_lemma}]
We are going to show that if there is a sequence $\Prb_N$ such that
the weak limit point $\Prb$ violates \eqref{no_giant_formula} then
for some $N$ we have \be\label{contradiction:too_much_fire} \E_N
\left(\Phi^N(T)\right)
> m_0(0) \ee which is in contradiction with
\eqref{theta_def_nonnegative_general_frozen_eq}.

We define $\rbeps_1$, $\rbeps_2$, $\rbeps_3>0$ and $\spect$ using
Lemma \ref{lemma:disjoint_intervals_big_mass}.
Next, we define
$\rbeps_4$ using this $\rbeps_2$ and Lemma
\ref{lemma:if_big_mass_big_fire_soon}. Given these, we choose
$\rbt$ be so small that
\[
\left\lfloor \frac{\rbeps_1}{2\rbt} \right\rfloor \rbeps_3
\rbeps_4 >m_0(0).
\]
We choose $K$ and $N_1$ big enough so that
\eqref{if_big_mass_big_fire_soon_eq} holds for this $\rbt$.
Further on, we fix the intervals $\lbrack \alpha_i,\beta_i
\rbrack$, $1 \leq i \leq m=\lfloor \frac{\rbeps_1}{2\rbt} \rfloor$
so that $\alpha_{i+1}-\beta_i > \rbt$ holds for all $i$ and also
$T -\beta_m>\rbt$ holds. We choose $N_0$ such that
\eqref{disjoint_intervals_big_mass} holds and let $N:=\max
\{N_0,N_1\}$.

Finally, we define the stopping times $\tau_1, \tau_2,\dots,
\tau_m$ by
\[
\tau_i:=\beta_i \wedge \min \{ t: t \geq \alpha_i \text{ and }
\sum_{k>K} v_k^N(t) \geq \rbeps_2 \}.
\]
We have $\tau_i+\spect \leq \beta_i +\spect < \alpha_{i+1} \leq
\tau_{i+1}$.

Using the strong Markov property,
\eqref{if_big_mass_big_fire_soon_eq} and
\eqref{disjoint_intervals_big_mass}, the inequality
\eqref{contradiction:too_much_fire} follows:
\begin{multline*}
\expect{\Phi^N(T)}\geq \sum_{i=1}^m
\expect{\Phi^N(\tau_i+\spect)-\Phi^N(\tau_i)} \geq \\
 \sum_{i=1}^m \expect{
 \condexpect{(\Phi^N(\tau_i+\spect)-\Phi^N(\tau_i))\ind \lbrack
 \sum_{k>K} v_k^N(\tau_i) \geq \rbeps_2 \rbrack }{\mathcal{F}_{\tau_i} }
 } \geq \\
 \sum_{i=1}^m \rbeps_4 \prob{\sum_{k>K} v_k^N(\tau_i) \geq
 \rbeps_2}
\geq m \rbeps_4 \rbeps_3>m_0(0).
\end{multline*}
\end{proof}

For a frozen percolation evolution defined by
\eqref{frozen_evolution_for_N} we have \be
\label{def_U_for_finite_N}
 U(t,x)=\sum_{k \geq 1}
v_k^N(t) e^{-kx} -(m_0(0)-\Phi^N(t))=V(t,x)-m_0^N(t)=\sum_{k \geq
1} v_k^N(t)\left(e^{-kx}-1 \right) \ee

We will make use of the following generating function estimates in
the proof of Lemma \ref{lemma:almost_giant_or_lot_burnt}.

 If
$U(x)=\sum_{k \geq 1}v_k \left(e^{-kx}-1 \right)$ where $\conf{v}
\in \myN$ then
\begin{eqnarray}
\label{tail_to_gf}
\sum_{k > K} v_k  \geq \varepsilon
& \implies &
U(1/K) \leq (e^{-1}-1)\varepsilon
\\
\label{gf_to_tail}
U(1/K) \leq -\varepsilon
& \implies &
\sum_{k >\frac{\varepsilon K}{2}}  v_k  \geq \varepsilon/2.
\end{eqnarray}

\begin{lemma}
\label{lemma:almost_giant_or_lot_burnt} There are constants $C_1
<+\infty$, $C_2>0$, $C_3>0$ such that if
\begin{equation}
\label{bigmass66} \sum_{k>K} v_k^N(0) \geq \rbeps_2
\end{equation}
for all $N$ then
\begin{equation}
\label{almost_giant_or_lot_burnt_eq} \lim_{N \to \infty} \prob{
\sum_{k>C_3 \rbeps_2 N^{1/3}} v_k^N \left(\rbtfin \right)+ \Phi^N
\left(\rbtfin \right)
 \geq C_2 \rbeps_2} =1
\end{equation}
Where $\rbtfin=\frac{C_1}{K\rbeps_2}$.
\end{lemma}

\begin{proof}[Sketch proof]
If we let $N \to \infty$ immediately, then by Lemma \ref{lemma_weaklimit} we get that the limiting functions $v_1(t),v_2(t),\dots$
solve \eqref{general_frozen_ODE} with initial condition $\conf{v}(0)$,
 a possibly random control function $\Phi(t)$ and some nonnegative  rate function  $\lambda(t)$.

The $N \to \infty$ limit of \eqref{almost_giant_or_lot_burnt_eq} is
\begin{equation}\label{toy}
 \theta \left( \rbtfin \right)+ \Phi\left(\rbtfin\right) \geq C_2
\rbeps_2 \end{equation}
Now we prove that if $\conf{v}(\cdot)$ is a solution of \eqref{general_frozen_ODE}
 then
 $\sum_{k>K} v_k(0) \geq \rbeps_2$  implies \eqref{toy} with $C_1=4$ and $C_2=\frac{1}{4}$.
  This proof will also serve as an outline of the proof of Lemma \ref{lemma:almost_giant_or_lot_burnt}.

In order to prove \eqref{toy} define $V(t,x)$ by \eqref{gfdeffff}. Thus $V(t,x)$ solves
\begin{equation}\label{V_diffe}
\dot{V}(t,x)=V'(t,x)\cdot \left( m_0(0)-\Phi(t)+\lambda(t) - V(t,x)\right)
\end{equation}

Define $U(t,x)$ by \eqref{def_U}. Define the characteristic curve $x(\cdot)$ by
\begin{equation}\label{charadiff_sketch}
\dot x(t)=V(t,x(t))-\left( m_0(0)-\Phi(t) +\lambda(t)\right) \qquad x(0)=\frac1K
\end{equation}
Let $\rbu(t):=V(t,x(t))$. Now by  \eqref{V_diffe} and \eqref{charadiff_sketch} we get
\begin{equation}\label{chara_u_ineq_sketch}
 \dot \rbu(t)=\dot{V}(t,x(t))+V'(t,x(t))\dot{x}(t)=0
  \end{equation}
 Thus $\rbu(t) \equiv \rbu(0)$, moreover by \eqref{def_U} we get
 $
 U(t,x(t))-U(0,x(0))=\Phi(t)
 $ and by $V(t,x(t))\equiv V(0,x(0))$, $V(0,x(0))-m_0(0)=U(0,x(0))$ and \eqref{charadiff_sketch} we get
 \begin{equation}\label{charaint_sketch}
x(t)=\frac1K +\int_{0}^t \Phi(s)\,ds
- \int_{0}^t \lambda(s)\,ds + t \cdot U(0,\frac1K)
\end{equation}
By \eqref{tail_to_gf} we have $U(0,\frac1K) \leq -\frac12 \rbeps_2$. In order to prove that
$\theta \left( \rbtfin \right)+ \Phi\left(\rbtfin\right) \geq \frac14 \rbeps_2$
with $\rbtfin=\frac{4}{K \rbeps_2}$ we consider two cases:

If $\Phi \left( \rbtfin \right) \geq \frac14 \rbeps_2$ then we are done.
If $\Phi \left( \rbtfin \right) < \frac14 \rbeps_2$ define $\tau:=\min \{t: x(t)=0\}$.
By \eqref{charaint_sketch} we have
\[ x(\rbtfin) \leq \frac1K + \rbtfin\cdot \Phi(\rbtfin)+ \rbtfin\cdot \left(-\frac12 \rbeps_2 \right) <
 \frac1K +\frac1K- \frac2K=0 \]
Thus $\tau \leq \rbtfin$.
\[ -\theta(\tau)=
U(\tau,0)=U(\tau,x(\tau))=U(0,\frac1K)+\Phi(\tau) \leq -\frac12 \rbeps_2 + \frac12 \rbeps_2 =-\frac14 \rbeps_2
\]
Thus $\frac14 \rbeps_2 \leq \theta(\tau) \leq \theta(\tau)+\Phi(\tau)\leq
\theta \left( \rbtfin \right)+ \Phi\left(\rbtfin\right)$ because by \eqref{theta_def_nonnegative_general_frozen_eq} the function
 $\theta(t)+\Phi(t)$ is increasing.
\end{proof}
To make this proof work for Lemma \ref{lemma:almost_giant_or_lot_burnt}  we have to deal with the fluctuations caused by randomness and combinatorial error terms.

\begin{proof}
Given a frozen percolation evolution obtained from a Markov
process by \eqref{frozen_evolution_for_N} define $U$ and $V$ by
\eqref{def_U_for_finite_N}.

Using \eqref{frozen_infinitesimal_expected_change_v}
 a straightforward calculation shows that
\begin{multline}
\label{V_infinitesimal} L V(t,x):= \lim_{h \to 0_+}
\frac{1}{h}\condexpect{V(t+h,x)-V(t,x)}{\cF_t}= \\
V'(t,x) \left( (m_0(0)-\Phi^N(t))+\lambda(t)\mu(N)-V(t,x) \right)
 + \frac{1}{N} \left( V^{\prime \prime} (t,x)-V^{\prime
\prime}(t,2x) \right)
\end{multline}
 Given the random
function $V(t,x)$ we define the random characteristic curve $x(t)$
similarly to \eqref{charadiff_sketch}:
\begin{equation}
\label{rnd_charadiff}
 \dot{x}(t)=V(t,x(t))-\left((m_0(0)-\Phi^N(t))+\lambda(t)\mu(N) \right),
 \quad \quad x(0)=\frac{1}{K}
\end{equation}
This ODE is well-defined although $V(t,x)$ is not continuous in
$t$, but almost surely it is a step function with finitely many
steps which is a sufficient condition to have well-posedness for
the solution of \eqref{rnd_charadiff}. Define
$\rbu(t):=V(t,x(t))$.
\begin{multline}
\label{char_rnd_int} x(t)=\frac1K+\int_{0}^t
\left(\rbu(s)-\rbu(0)\right)ds +\int_{0}^t \Phi^N(s)
ds
-\mu(N) \int_{0}^t \lambda(s)ds + t \cdot U(0,\frac1K)
\end{multline}
Putting together \eqref{V_infinitesimal} and \eqref{rnd_charadiff}
we get
\begin{equation}
\label{char_infinitesimal}
 \lim_{h \to 0_+}
 \frac{1}{h}\condexpect{\rbu(t+h)-\rbu(t))}{\cF_t}=
 \frac{1}{N}
\left( V^{\prime \prime} (t,x(t))-
V^{\prime\prime}(t,2x(t))\right) = \Ordo \left( \frac{1}{N}
V^{\prime\prime} (t,x(t)) \right)
\end{equation}
Thus $\wt\rbu(t)=\rbu(t)-\int_{0}^t \frac{1}{N} \left( V^{\prime
\prime} (s,x(s))-V^{\prime \prime}(s,2x(s))\right)ds$ is a
martingale and by \eqref{lightning_rate} and \eqref{coag_rate} we
get
\begin{multline}
\label{char_variance} \lim_{h \to 0_+} \frac{1}{h} \condexpect{\wt
\rbu(t+h)^2- \wt\rbu(t)^2}{\cF_t} = \lim_{h \to 0_+}\frac{1}{h}
\condexpect{\big(V(t+h,x(t))-V(t,x(t))\big)^2 }{\cF_t}
\\
\leq \frac{1}{2} \sum_{k,l=1}^N \left(\frac{k+l}{N} e^{-(k+l)x(t)}
-\frac{k}{N} e^{-k x(t)} -\frac{l}{N}e^{-lx(t)}\right)^2
v_k^N(t)v_l^N(t) N
\\
+ \sum_{l=1}^N \left(\frac{l}{N} e^{-lx(t)}\right)^2
\mu(N)\lambda(t) v_l^N(t) N = \Ordo \left(\frac{1}{N} V^{\prime
\prime} (t,x(t)) \right)
\end{multline}

 Define the stopping time
\[
\tau_N:=\min\{t:x(t)=N^{-1/3}\}.
\]
(Note that we could replace $N^{-1/3}$ by $N^{-\gamma}$, $0<\gamma<1/2$ without changing the proof.)

It follows from  \eqref{trivibounds}, \eqref{char_infinitesimal},
\eqref{char_variance} and
 Doob's maximal inequality that
\begin{equation}
\label{char_weak} \sup_{0\leq t \leq T }\abs{\rbu(t \wedge \tau_N
\wedge T)-\rbu(0)} \weak 0 \quad\text{ as }\quad  N \to \infty
\end{equation}
By \eqref{tail_to_gf} and \eqref{bigmass66} we have
\begin{equation}
\label{rbeps5} U(0,x(0))
 \leq
 (e^{-1} -1)\rbeps_2=:-\rbeps_5
\end{equation}

Let
\begin{align*}
A_N &:= \big\{\int_{0}^{\tau_N\wedge T} \abs{ \rbu(s) -\rbu(0)}ds
\leq \frac{1}{K} \big\} \cap \big\{\abs{\rbu(\tau_N\wedge
T)-\rbu(0)} \leq {\rbeps_5}/3 \big\},
\\[8pt]
B_N &:= \big\{\Phi^N(\tau_N) \leq {\rbeps_5}/3\big\}.
\\[8pt]
\rbtfin &:= \frac{3}{K \abs{U(0, x(0))}} \leq \frac{3}{K\rbeps_5},
\end{align*}
 We
are going to show that that there are constants $C_2,C_3<+\infty$
such that
\begin{equation}
\label{we_show_this} A_N \subseteq \left\{ \sum_{k>C_3 \rbeps_2
N^{1/3}} v_k^N \left(\rbtfin \right) +  \Phi^N \left(\rbtfin
\right)
 \geq C_2 \rbeps_2
\right\}
\end{equation}
which together with \eqref{char_weak} implies $\lim_{N \to \infty}
\prob{A_N}=1$ and \eqref{almost_giant_or_lot_burnt_eq}.

First we show that
\begin{equation}
\label{ABsubseteq} A_N \cap B_N \subseteq \{\tau_N \leq  \rbtfin
\}.
\end{equation}
If we assume indirectly that $A_N$, $B_N$ and $\tau_N>\rbtfin$
hold then $\int_{0}^{\rbtfin}
\abs{\rbu(s)-\rbu(0)}ds\leq\frac{1}{K}$, so by
\eqref{char_rnd_int} we get
\[
x(\rbtfin) \leq \frac{1}{K} +\frac{1}{K}+ \int_{0}^{\rbtfin}
\Phi^N(s) ds+\rbtfin U(0,x(0)) \leq -\frac{1}{K}+ \rbtfin
\frac{\rbeps_5}{3} \leq 0.
\]
But $x(\rbtfin) \leq 0$ is in contradiction with $\tau_N>\rbtfin$,
thus \eqref{ABsubseteq} holds. Assuming $A_N$ and $B_N$ we obtain
\[
\abs{ U(\tau_N,x(\tau_N))-U(0,x(0))}\leq
\abs{\rbu(\tau_N)-\rbu(0)} + \Phi^N(\tau_N) \leq \rbeps_5/3+
\rbeps_5/3
\]
which together with \eqref{rbeps5} implies $A_N \cap B_N \subseteq
\{ U(\tau_N,N^{-1/3}) \leq -\rbeps_5/3 \}$

 By \eqref{gf_to_tail}
\begin{align*}
A_N \subseteq (A_N \cap B_N)\cup B_N^c &\subseteq
\big\{\sum_{k>N^{1/3}{\rbeps_5}/{6}} v_k^N(\tau_N) \geq
{\rbeps_5}/{6}\big\} \cup \big\{\Phi^N(\tau_N) > \rbeps_5/3\big\}
\\[8pt]
&\subseteq \left\{ \sum_{k>C_3 \rbeps_2 N^{1/3}} v_k^N(\tau_N) +
\Phi^N(\tau_N) \geq C_2 \rbeps_2 \right\}
\end{align*}
with $C_2=C_3=(1-e^{-1})/6$. But $\sum_{k>C_3\rbeps_2 N^{1/3}}
v_k^N(t)+ \Phi^N(t)$ is a monotone increasing function of $t$,
from which \eqref{we_show_this} follows.
\end{proof}

\begin{lemma}
\label{lemma:if_almost_giant_then_big_fire_soon} There are
constants $C_4<+\infty$, $C_5>0$ such that if
\[
\sum_{k>C_3 \rbeps_2 N^{1/3}} v_k^N(0)\geq {C_2\rbeps_2}/{2}
\]
for all $N$ then with
\begin{equation}
\label{t_bla_def}
\rbtfin_N:=C_4\rbeps_2^{-2}\big(N^{-1/3}\log(N)+(N\mu(N))^{-1}\big)
\end{equation}
we have
\begin{equation}
\label{bla_bla_7} \lim_{N \to \infty} \expect{\Phi^N(\rbtfin_N) }
\geq C_5 \rbeps_2.
\end{equation}
\end{lemma}
\begin{remark}
The upper bound \eqref{t_bla_def} is technical: on one hand it is not optimal,
 on the other hand, for the proof of
Lemma \ref{lemma:if_big_mass_big_fire_soon} we only need $\rbtfin_N \ll 1$ as $N \to \infty$.
\end{remark}
\begin{proof}
If $v$ is a vertex of the graph $G(N,t)$ let $\cC_N(v,t)$ denote
the connected component of $v$ at time $t$. Denote by $\tau_b(v)$
the freezing/burning time of $v$.

%For $t \geq \tau_b(v)$ we define $\cC_N(v,t):=\lim_{h \to 0_+}
%\cC_N(v,\tau_b(v)-h)$.

 \[ \cH_N(t):= \{ v  \, : \,
 \abs{\cC_N(v,0)} \geq C_3 \rbeps_2 N^{\frac13}\, \text{ and }\, \tau_b(v)>t  \} \]
% We are going to focus on $\cH_n(t)$: these vertices may coagulate
%with smaller vertices as time evolves, but that increases both the
%rate with which forest fires burn them and the amount of forest
%burnt, thus the lower bounds we are going to give remain valid if
%we add smaller components.
We fix a vertex $v \in \cH_N(0)$.
\begin{align*}
c_N(t)&:=\frac{1}{N} \abs{\cC_N(v,t)}\\
w_N(t)&:=
\frac{1}{N}\abs{\cH_N(t)} \\
z_N(t)&:=\frac{1}{N}\abs{\cH_N(0)\setminus \cH_N(t)}=
w_N(0)-w_N(t)
\end{align*}
 Thus $c_N(t)$ is an increasing process until $\tau_b(v)$,
 $w_N(t)$ is decreasing, $z_N(t)$ is increasing.
  We consider the right-continuous version of the processes $c_N(t), w_N(t), z_N(t)$.
\[w_N(0) \geq {C_2\rbeps_2}/{2}
 =:\rbeps_6.\]
  We are going to prove that there are constants $C_4<+\infty$,
$C_5>0$ such that
\begin{equation}
\label{we_show_this_2} \lim_{N \to \infty} \expect{z_N(\rbtfin_N)}
\geq C_5 \rbeps_2
\end{equation}
with $\rbtfin_N$ defined as in \eqref{t_bla_def}. This implies
\eqref{bla_bla_7}.

 We define the stopping times
\begin{align*}
 \tau_w&:= \min \{t: w_N(t) < {\rbeps_6}/{2}
\} \\ \tau_g&:= \min \{t: c_N(t) > {\rbeps_6}/{4} \} \\
\tau&:=\tau_b(v) \wedge \tau_w \wedge \tau_g
 \end{align*}
 Let $\bar{N}:=C_3 \rbeps_2 N^{\frac13}$.
 Since $v \in \cH_n(0)$ we have
 \[c_N(t) \geq c_N(0) =\frac{\abs{\cC_N(v,0)}}{N} \geq \frac{\bar{N}}{N}
 \]

 If $\cC_N(v,t)$ is
 connected to a vertex in $\cH_N(t)$ by a new edge at time $t$
 then
\[
   c_N(t_+)-c_N(t_-) \geq \frac{\bar{N}}{N}, \quad \log(c_N(t_+))-
   \log(c_N(t_-))\geq \log \left( 1+ \frac{\bar{N}}{N c_N(t_-)} \right) \geq
   \frac{\log(2)\bar{N}}{N c_N(t_-)}
   \]
\begin{multline*}
\lim_{dt \to 0}\frac{1}{dt}
\condexpect{\log(c_N(t+dt))-\log(c_N(t))}{\cF_t}\geq \\
\frac{\log(2)\bar{N}}{N c_N(t)} \lim_{dt \to 0}\frac{1}{dt}
\condprob{c_N(t+dt)-c_N(t)\geq \frac{\bar{N}}{N} }{\cF_t}\geq \\
\frac{\log(2)\bar{N}}{N c_N(t)}\cdot
\frac{1}{N}\abs{\cC_N(v,t)}\left(\abs{\cH_N(t)}-
\abs{\cC_N(v,t)}\right)\ind_{\{t\ \leq \tau_b(v)\}}
\geq \\
\log(2)\bar{N} \cdot \left( w_N(t)-c_N(t)\right) \ind_{\{ t\ \leq \tau_b(v)\} } \geq \\
 \log(2)\bar{N}\frac{\rbeps_6}{4}\ind_{ \{ t \leq \tau\} }= N^{1/3} \frac{\log(2)}{8} \cdot C_2 \cdot C_3 \cdot
 (\rbeps_2)^2\cdot \ind_{ \{ t \leq \tau\} }=:a \cdot \ind_{ \{t \leq \tau \}}
\end{multline*}
Thus $\log(c_N(t))- a \cdot (t \wedge \tau)$ is a submartingale.  Using the
 optional sampling theorem we get
\[
\log(m_0(0))-a \cdot \expect{\tau}\geq \expect{\log(c_N(\tau))} -
a\cdot \expect{\tau} \geq
  \log(c_N(0)) \geq
 -\log(N)
 \]
By Markov's
 inequality  we obtain that for some constant $C<+\infty$
\[
\prob{\tau \leq C N^{-1/3} \rbeps_2^{-2} \log(N)} \geq \frac12
\]
if $N$ is sufficiently large.

If $\tau_g \leq \tau_b(v)$, then $\cC_N(v,\tau_g)
>\frac{\rbeps_6}{4} N $, so $\expect{\tau_b(v) -\tau_g} \leq
(N\mu(N)\lambda_{inf})^{-1} \frac{4}{\rbeps_6}$, which implies
\[
\prob{\tau_w \wedge \tau_b(v) \leq C N^{-1/3} \rbeps_2^{-2}
\log(N) + C^{\prime} (N\mu(N))^{-1} \rbeps_2^{-1}} \geq
\frac{1}{4}.
\]
for some constant $C^{\prime}$. Define  $\rbtfin$ of
\eqref{t_bla_def} with $C_4:=\max\{C,C^{\prime} \}$. Using the
linearity of expectation we get
\[\expect{z(\rbtfin)}=\expect{\frac{1}{N} \sum_{w \in \cH_N(0)}
\ind_{\{\tau_b(w) \leq \rbtfin\}} } \geq \rbeps_6 \prob{\tau_b(v)
\leq  \rbtfin}.
\]
The inequality
 $\ind_{\{ \tau_w \leq \rbtfin \}}\frac{\rbeps_6}{2}  \leq
 z(\rbtfin)$ follows from the definition of $\tau_w$.
\[
\frac{1}{4} \leq \prob{\tau_w \wedge \tau_b(v) \leq \rbtfin} \leq
\prob{ \tau_w \leq \rbtfin} + \prob{\tau_b(v) \leq \rbtfin} \leq
\expect{z(\rbtfin)} \frac{2}{\rbeps_6}+ \expect{z(\rbtfin)}
\frac{1}{\rbeps_6}
\]
From which \eqref{we_show_this_2} follows.
\end{proof}

%\begin{proof}[Proof of Lemma \ref{lemma:if_big_mass_big_fire_soon}.]
%\ref{lemma:almost_giant_or_lot_burnt}. and
%\ref{lemma:if_almost_giant_then_big_fire_soon}.

%Given $\rbeps_2$ and $\rbt$ choose $K$ and $N_1$ such that for $N
%\geq N_1$ we have
%\[ \sum_{k>K} v_k^N(0) \geq \rbeps_2 \quad \implies \quad \prob{
%\sum_{k>C_3 \rbeps_2 N^{1/3}} v_k^N \left(\rbtfin \right)+ \Phi^N
%\left(\rbtfin \right)
% \geq C_2 \rbeps_2} \geq \frac12 \]
%\[ \sum_{k>C_3 \rbeps_2 N^{1/3}} v_k^N(0)\geq {C_2\rbeps_2}/{2}
%\; \implies  \; \lim_{N \to \infty} \expect{ \Phi^N \left(
%C_4\rbeps_2^{-2}\left(\frac{\log(N)}{N^{1/3}}+\frac{1}{N\mu(N)}\right)
%\right)} \geq \frac{C_5 \rbeps_2}{2}\]
%\[\frac{C_1}{K \rbeps_2} +
%C_4\rbeps_2^{-2}\big(N^{-1/3}\log(N)+(N\mu(N))^{-1}\big) \leq \rbt
%\]
 Lemma \ref{lemma:if_big_mass_big_fire_soon}. is a straightforward consequence
 of  Lemma \ref{lemma:almost_giant_or_lot_burnt}. and Lemma
\ref{lemma:if_almost_giant_then_big_fire_soon}.

\section{Properties of the solutions of the frozen percolation
equations}\label{section_properties_of_frozen_eqs}

\begin{proof}[Proof of Theorem  \ref{tauberian_smol}.]
It is clear from \eqref{crit_varphi_def_formula} and
\eqref{Lipschitz} that $\varphi(t)$ is continuous. In order to
prove \eqref{powerlaw_decay} we need Example (c) of Theorem 4. of
chapter XIII.5 of \cite{feller}. By (\ref{E_solution_translated})
%\[ \int_0^{\infty} e^{-y\cdot x}\sum_{k>y} v_k(t)dy=-\frac{U(t,x)}{x}
%\]
\[\kderi{X}{u}(t,0)=\frac{1}{E(t,0)}=\frac{1}{E(0,t)}=\frac{1}{\varphi(t)}\]
% by (\ref{E_solution_translated}), so
\[X(t,u)=\frac{1}{2 \varphi(t)} u^2+ \Ordo(u^3),\quad \lim_{x \to 0}
\frac{-U(t,x)}{\sqrt{x}}=\sqrt{2 \varphi(t)}\]
By the
Tauberian theorem for any $t \geq \gt$  each of the relations
\[-U(t,x)\sim x^{1-1/2} \sqrt{2\varphi(t)} \quad \text{and} \quad
\sum_{k=K}^{\infty} v_k(t) \sim \frac{1}{\Gamma(\frac{1}{2})}
K^{1/2-1} \sqrt{2\varphi(t)} \] implies the other, that is for any $t \geq \gt$
\[ \lim_{x \to 0}
\frac{-U(t,x)}{\sqrt{x}}=\sqrt{2 \varphi(t)} \quad \iff \quad
\lim_{K \to
\infty} K^{\frac12}\sum_{k=K}^{\infty} v_k(t)=\sqrt{\frac{2
\varphi(t)}{\pi}}
\]
\end{proof}

In order to compare the solutions of (\ref{alter_smol_frozen_eq})
and (\ref{crit_smol_frozen_integral_eq}) we apply the
transformations \be \label{transformations_UXG} \conf{v}(t) \to
U(t,x) \to X(t,u) \to G(t,w) \ee to the solutions of the
alternating equations:
% When we define $U$ by (\ref{def_U}), we get
%$V(0,0)-V(t,0)=\Phi(t)+\theta(t)$ by
% Definition \ref{def_sub_crit_smol_equations},
%thus we have $U(t,0)=
 %-\theta(t)$.

  The integral equation \be \label{alter_U_integral_eq} U(t,x)=
U(0,x)+\int_0^t -U(s,x)\deri{U}{x}(s,x)ds +\Phi(t)
 \ee
holds, but $\Phi(t)$ is constant between burning times and jumps
by $\theta(\bt_i)$ at $\bt_i$, which means that the giant
component is burnt:
\[\lim_{\varepsilon \to 0} -U(\bt_i+\varepsilon,0)=
  \lim_{\varepsilon \to 0}  \theta(\bt_i
+\varepsilon)= \theta(\bt_{i+})=0\] By Lemma
\ref{lemma_X_solution}. the formulae
(\ref{X_solution_inhom_subcrit}),
(\ref{G_solution_inhom_subcrit}), (\ref{F_solution_inhom_subcrit})
and (\ref{E_solution_translated}) are valid (with rate function
$\lambda(t)\equiv 0$).

 In between the burning times $\bt_i< t  \leq \bt_{i+1}$ we have
\[X(t,u)=X(\bt_{i+},u)+(t-\bt_i)u \quad \text{and}\quad
G(t,w)=G(\bt_{i+},w+(t-\bt_i)).\] If $t-\bt_i>w^*(\bt_{i+})$ then
$\conf{v}(t)$ is supercritical:
\[\deri{X}{u}(t,0)>0,\quad \theta(t)>0,\quad X(t,-\theta(t))=0, \quad
\deri{X}{u}(t,-\theta(t))<0.\] $\min_w G(t,w)=0$ still holds, but
$\text{argmin}_w G(t,w)=w^*(t) <0$ in the supercritical phase. Thus
$-\deri{X}{u}(t,0)=w^*(t)$ is well-defined for all $t \geq 0$ for
the solutions of the equations (\ref{subcrit_smol_frozen_integal_eq}),
(\ref{crit_smol_frozen_integral_eq}) and (\ref{alter_smol_frozen_eq}) as
well, moreover (\ref{x_star_and_X_G}) holds. For the solutions of (\ref{alter_smol_frozen_eq}) $w^*(t)$ is left-continuous.

By $G(t,w^*(t))\equiv 0$, (\ref{G_solution_inhom_subcrit}) and (\ref{Phi_and_F}) we get
\be\label{alter_int_Phi_formula}
\int_0^t \Phi(s)ds =G(0,t+w^*(t))-w^*(t)F(0,t+w^*(t))
\ee
for the solutions of (\ref{alter_smol_frozen_eq}).

If $\conf{v}(t)$ is the solution of
(\ref{subcrit_smol_frozen_integal_eq}), (\ref{crit_smol_frozen_integral_eq}) or
(\ref{alter_smol_frozen_eq}) started from
$\conf{v}(0) \in \myNz$, then
(\ref{E_solution_translated}) holds: the evolution of the critical core
 does not depend on the rate of lightnings. One extra parameter is needed to
determine $\conf{v}(t)$ and $\theta(t)$: if we
know $w^*(t)$, then \be \label{E_and_w_star_determines_everything}
F(t,w)=\int_{w^*(t)}^w E(0,t+y)dy \quad \text{ and }\quad
G(t,w)=\int_{w^*(t)}^w (w-y)E(0,t+y)  dy \ee
 has all the
information about $\conf{v}(t)$ and $\theta(t)$, since the
transformations (\ref{transformations_UXG}) are invertible (using
analytic extensions).

 \begin{proof}[Proof of Claim \ref{selfsimilarity}.]
 First assume $m_0(0)=1$.
As a consequence of Remark \ref{monodisperse_E_remark}. we can see
that \be \label{selfsimilar_core}
E(t,w)=E(0,t+w)=\frac{1}{(w+t)^2}=\frac{1}{t^2}E(1,\frac{w}{t}),\ee
but this is the critical core of $\frac{1}{t}\conf{v}(1)$, and
together with $w^*(t) \equiv 0$ for $t \geq \gt =1$ the identity
$v_k(t)=\frac{1}{t}v_k(1)$ follows.
 We get the explicit formula for $v_k(1)$ in the following way: since $X(1,u)=X(0,u)+u$,
 the inverse function of $V(1,x)$ is $-\log(v)+v-1$, thus
\[V(1,x)=-W\left(-e^{-(x+1)}\right)=\sum_{k=1}^{\infty} \frac{k^{k-1}}{k!} e^{-k}
e^{-kx}\] where $W$ is the Lambert W function, the inverse
function of $z \mapsto ze^z$.

 If $m_0(0)\neq 1$ but we still have
a monodisperse initial condition then \eqref{selfsimilar_core}
still holds and for $t \geq \frac{1}{m_0(0)}=\gt$ we have
$w^*(t)=0$ thus $v_k(t)=\frac{1}{t} \frac{k^{k-1}}{k!}e^{-k}$ must
hold.
\end{proof}

\begin{proof}[Proof of Theorem \ref{asymptotic_selfsimilar}.]

% $\hat{I}(x)=\log(V(0,-x))$ is the logarithmic moment generating function of $\%left( v_k(0) \right)_{k=1}^{\infty}$, so $\hat{I}'(x)$ is increasing and the ex%pected value of the exponentially tilted measures converge to the left endpoint% of the support of the measure:
%  1=\lim_{x \to \infty} -\frac{V'(0,x)}{V(0,x)}
% since $v_1(0)>0$.
 Let $H(w):=F(0,w)-m_0(0)$, thus $H \left(-\frac{1}{V'(0,x)}\right)=-V(0,x)$ by
(\ref{F_U_E_U_formulas}). Using Lemma \ref{crit_eq_existence}. and
\eqref{F_solution_inhom_subcrit} we get
 \[F(t,w)=H(t+w)-H(t) \quad
\text{ and } \quad  m_0(t)=F(t,+\infty)=-H(t)\] for $t \geq \gt$.
 $v_1(0)>0$ implies $\lim_{x \to \infty} -\frac{V'(0,x)}{V(0,x)}=1$, so $\lim_{t \to \infty} t \cdot H(tw)=-\frac{1}{w}$, from which $\lim_{t \to \infty} t m_0(t)=1$ follows. Moreover
 \[1-\frac{1}{w+1}=\lim_{t \to \infty} t \cdot \left( H(t\cdot(w+1))-H(t)\right)=
\lim_{t\to \infty} t \cdot F(t,tw)=\lim_{t \to \infty}
\hat{F}(t,w) \] where $\hat{v}_k(t)=t v_k(t)$. This implies the
pointwise convergence of the monotone functions $\hat{X}'(t,u)$,
$\hat{X}(t,u)$, $\hat{U}(t,x)$ and $\hat{V}(t,x)$ to the desired
limit as $t \to \infty$. The convergence of $\hat{v}_k (t)$ to
$\frac{k^{k-1}}{k!}e^{-k}$ follows from the continuity theorem of
Laplace transforms.
\end{proof}

\begin{proof}[Proof of Theorem \ref{exp_tilt_rigidity}.]
 It is easy to check that if $\tilde{v}_k(t)=v_k(t) e^{-k x^*(t)}$,
  then $\tilde{V}(t,x)=V(t,x+x^*(t))$, so $\tilde{x}^*(t)=0$ and
   $\tilde{w}^*(t)=0$, but $\tilde{E}(t,w)=E(0,t+w)=E(t,w)$, so
    $\tilde{\conf{v}}(t)$ is identical to the solution of
    (\ref{crit_smol_frozen_integral_eq}) at time $t$.
\end{proof}

\begin{proof}[Proof of Theorem \ref{extremum_theorem}.]
 If we
consider the solution of (\ref{subcrit_smol_frozen_integal_eq})
with given  initial data and lightning rate function $\lambda(t)
\geq 0, 0\leq t\leq T$ then (\ref{G_solution_inhom_subcrit})
provides us with a relation between our cost ($\int_0^T
\lambda(t)dt$) and reward ($\int_0^T \Phi(t)dt$).

We prove \eqref{subcrit_functional_extremum} by considering the cases
$T\geq \gt$ and $T \leq \gt$ separately.

 According to
(\ref{kritikus_celfuggveny}), for $T\geq \gt$ we get \[0 \leq
G^{sub}(T,0)=\int_0^T \Phi^{crit}(t)dt -\int_0^T \Phi^{sub}(t)dt+\int_0^T
\lambda(t)dt\] by substituting $w=0$ into
(\ref{G_solution_inhom_subcrit}).

 For $T \leq \gt$, we want to
prove $0 \geq \int_0^T \Phi^{sub}(t)dt -\int_0^T \lambda(t)dt$.
Substitute $w=\gt-T$ into (\ref{G_solution_inhom_subcrit}). Since
$G(0,\gt)=0$ and $(\gt-T)\Phi^{sub}(T) \geq 0$ we get \[0 \leq
G^{sub}(T,\gt-T)\leq -\int_0^T \Phi^{sub}(t)dt +\int_0^T
\lambda(t)dt.\]

 The proof of the extremum
property (\ref{alter_functional_extremum})  is equally simple.
\end{proof}

If we want to maximize our cost functional for a fixed $T>\gt$, the
optimal control is not unique, since the only thing we need for
\be \label{equality_in_extemum_problem}\int_0^T \Phi^{sub}(t)dt-\int_0^T
\lambda(t)dt=\int_0^T \Phi^{crit}(t)dt \ee
 to hold is $G^{sub}(T,0)=0$: if
$\conf{v}(T)$ is critical at time $T$, then
the value of the functional is optimal.

\begin{proof}[Proof of Remark \ref{sharp}.] In order to prove
(\ref{subcritical_functional_sharp})  first pick an arbitrary
$\lambda>0$ and solve (\ref{subcritical_control_ODE}) with
constant $\lambda(t)=\lambda$. Since $w^*(t)>0$ and $w^*(0)=\gt$
there is a $0< t^*\leq T$ such that $w^*(t^*)=T-t^*$, and the
lightning rate function $\lambda(t)=\lambda \cdot \ind \lbrack t
\leq t^* \rbrack$ makes $T$ a critical time,  so
(\ref{equality_in_extemum_problem}) holds, thus
\eqref{subcritical_functional_sharp}.

 Now we prove (\ref{alternating_functional_sharp}). By using
 (\ref{alter_int_Phi_formula}) we have to show that
 \[ G(0,T+w^*(T))-(w^*(T)-\varepsilon)
 F(0,T+w^*(T))>\\G(0,T)+\varepsilon F(0,T) \]
Using $G(0,T+w^*(T))-G(0,T) > w^*(T)F(0,T)$ it is easy to see that
$0<w^*(T)\leq \varepsilon$ is sufficient for this to hold. If
there is a $\gt<t^*\leq T$ such that
$-\deri{X}{u}(t^*,-\theta(t^*))=T-t^* +\varepsilon$, then burning
the giant component at time $t^*$ we get
$-\deri{X}{u}(t^*_+,0)=T-t^* +\varepsilon$ and
$-\deri{X}{u}(T,0)=w^*(T)=\varepsilon$. If not, then burning at
time $T$ yields $0<-\deri{X}{u}(T,-\theta(T))=
w^*(T_+)<\varepsilon$.
\end{proof}

\section{Proof of the subcritical limit
theorem}\label{section_proof_of_sub}

In order to prove Theorem \ref{subcrit_limit_thm_statement}.,
 we need to know more about the solution of (\ref{subcritical_control_ODE}).

\begin{lemma}\label{lemma_Lambert_W}
If $y(t)$ is the solution of the differential equation
$\dot{y}(t)=\frac{c}{y(t)}-1$ with initial condition $y(0)=\gt$
and $t \geq \gt+c\log(\frac{\gt}{c})$ then $y(t) \leq 2c$.
\end{lemma}
\begin{proof}
 The solution of this differential equation is
\be \label{Lambert_megoldokeplet} y(t)=c \cdot \left(1+W
\left(\exp\left(\frac{\gt-t}{c}-1\right)\cdot
\left(\frac{\gt}{c}-1\right)\right)\right) \ee where $W$ is the
Lambert W function.
Thus $W(x)\leq x$ and our claim follows.
\end{proof}

%\begin{lemma} Let $w^*(t)$ be the solution of
%(\ref{subcritical_control_ODE}). Then for $\hat{t} = T_g+16
%\frac{m_2(0)}{m_1(0)^3} \lambda \log(\frac{1}{\lambda})$

% we have $\abs{w^*(\hat{t})-\frac{\lambda}{E(0,\hat{t})}} \leq d \lambda
%\log(\frac{1}{\lambda})$ with $d=?$.
%\end{lemma}
%\begin{proof}
%$w^*(t)+t$ is an increasing function of $t$ by
%(\ref{subcritical_control_ODE}). Let $t^*$ denote the time when
%$w^*(t)+t =T_g+ 17 \frac{m_2(0)}{m_1(0)^3} \lambda
%\log(\frac{1}{\lambda})$. We will show that $t^*>\hat{t}$. For $t
%\leq \min \{t^*,\hat{t} \}$ we have
%\end{proof}

\begin{lemma}\label{subcrit_w_star_almost_phi}
If $w^*(t)$ is the solution of \eqref{subcritical_control_ODE}
with constant $\lambda(t)\equiv \lambda \leq 1$ then there exist
$d_1$ and $d_2$ which depend only on $\conf{v}(0)$ and $T$ such
that
 \[ \gt+d_1 \lambda \log(\frac{1}{\lambda})\leq t \leq T \quad
 \implies \quad
\abs{w^*(t)-\frac{\lambda}{E(0,t)}} \leq d_2 \lambda^2.\]
\end{lemma}
\begin{proof}
We have a uniform a priori bound $w^*(t)\leq w^*_{max}$ for all
$\lambda \leq 1$ depending only on the initial data and $T$ by
\eqref{w_min_w_max}. Thus by Lemma
\ref{lemma_upper_lower_bound_on_E} we have
\[0<\phiinf:=\phiinf(T+w^*_{max}) \leq E(0,t+w^*(t))\] and
substituting this inequality into \eqref{subcritical_control_ODE}
we get
 \be
\label{upper_lower_bound_of_sc_control_ode_with_lambert}
\dot{w}^*(t) \leq \frac{\lambda}{w^*(t)\phiinf}-1 \ee
 Using Lemma
\ref{lemma_Lambert_W}. we get \[\hat{t} := \gt+
\frac{\lambda}{\phiinf} \log(\frac{\phiinf \gt}{\lambda}) \leq
t\leq T \quad \implies \quad w^*(t) \leq 2
\frac{\lambda}{\phiinf}.\] Define
\[z(t):=\frac{w^*(t)E(0,t+w^*(t))}{\lambda} -1\] Using
(\ref{subcritical_control_ODE}) we get \be \label{z_ODE}
\dot{z}(t)=-\frac{1}{w^*(t)}z(t)+\frac{\deri{E}{w}(0,t+w^*(t))}{E(0,t+w^*(t)}
\ee
 For
$\hat{t}\leq t\leq T$ we have
 \be \label{z_ODE_inequalities} -1\leq z(\hat{t})\leq 2 \frac{\phisup}{\phiinf}, \quad
  \frac{1}{w^*(t)} \geq \frac12
\frac{\phiinf}{\lambda},\quad
\abs{\frac{\deri{E}{w}(0,t+w^*(t))}{E(0,t+w^*(t)}}\leq
\frac{D}{\phiinf}\ee with the D of (\ref{Lipschitz}). Solving the
linear ODE \eqref{z_ODE} and using the inequalities
\eqref{z_ODE_inequalities} we get
\[\abs{z(t)} \leq 2 \frac{\phisup}{\phiinf} \exp\left(-\frac12
\frac{\phiinf}{\lambda}(t-\hat{t})\right)+\lambda
\frac{2D}{\phiinf^2}.\] Thus for $t \geq \hat{t}+
\frac{2}{\phiinf} \lambda \log\left(\frac{1}{\lambda}\right)$ we
have $\abs{z(t)} =\Ordo(\lambda)$, which implies
\[w^*(t)-\frac{\lambda}{E(0,t+w^*(t))}=\Ordo(\lambda^2).\] If we
combine this with \[\abs{
\frac{\lambda}{E(0,t+w^*(t))}-\frac{\lambda}{E(0,t)}} \leq
\lambda^2 2\frac{D}{\phiinf^3}\] the claim of the Lemma follows.
\end{proof}
From this $\lambda
m_1^{\lambda}(t)-E(0,t)=\varphi_{\lambda}(t)-\varphi_{crit}(t)=\Ordo(\lambda)$
follows which proves \eqref{subcrit_control_almost_crit}. Now we
prove \eqref{subcrit_limit_thm_distribution_formulation} using
Laplace transforms:

\begin{lemma}\label{lemma_subcrit_limi_thm}
Let $U_{\lambda}(t,x)$ be the solution of (\ref{inhom_subcrit_U})
with a fixed initial condition $U(0,x)$ obtained from
$\conf{v}(0) \in \myNz$  and
$\lambda(t)\equiv \lambda$.
 Then for any $t>\gt$ we have \be
\label{limit_thm_Laplace} \lim_{\lambda \to 0}
\frac{\deri{U}{x}_{\lambda}\left(t, \frac{\lambda^2}{2 E(0,t)}x
\right)}{\deri{U}{x}_{\lambda}\left(t, 0\right)
}=\frac{1}{\sqrt{1+x}} \ee
\end{lemma}
\begin{proof} Fix $\lambda>0$ and denote the solution of
(\ref{inhom_subcrit_U}) with $\lambda(t)\equiv \lambda$ by
$U(t,x)$.
 For all $t\geq 0$  we have
\[ \kderi{X}{u}(t,u) \geq \frac{1}{\phisup}\quad \implies\quad X(t,u)\geq
\frac{1}{2\phisup}u^2 \quad \implies \quad \abs{U(t,x)}=
\Ordo(\sqrt{x}).\]

We use the shorthand notation $E=E(0,t+w^*(t))$.
\[\deri{X}{u}(t,u)=-w^*(t)+\frac{u}{E}+\Ordo(u^2), \quad X(t,u)=-u
w^*(t)+\frac{u^2}{2 E}+\Ordo(u^3)\]
\begin{multline*}U(t,x)=E
w^*(t)-\sqrt{\left(E w^*(t)\right)^2+2E (x-\Ordo(U(t,x)^3))}=\\
E w^*(t)-\sqrt{\left(E w^*(t)\right)^2+2E x}+\Ordo(x)
\end{multline*}
\begin{multline*}\deri{U}{x}(t,x)=\frac{1}{\deri{X}{u}(t,U(t,x))}=\frac{1}{-w^*(t)+\frac{U(t,x)}{E}}+\Ordo(1)=\\
\frac{-1}{\sqrt{w^*(t)^2+\frac{2}{E}x}
+\Ordo(x)}+\Ordo(1)=\frac{-1}{\sqrt{w^*(t)^2+\frac{2}{E}x}}+\Ordo(1)
\end{multline*}
Because of Lemma \ref{subcrit_w_star_almost_phi}. we have
\[\lim_{\lambda \to 0}
\frac{\lambda^2}{E(0,t+w^*_{\lambda}(t))E(0,t)
w^*_{\lambda}(t)^2}=1\] from which the claim of this lemma
follows.

% \be
%\frac{\deri{U}{x}_{\lambda}\left(t, \frac{\lambda^2}{2 E(0,t)}x
%\right)}{\deri{U}{x}_{\lambda}\left(t, 0\right)
%}=\frac{1}{\sqrt{1+\left(
%\frac{\lambda^2}{E(0,t)E(0,t+w^*_{\lambda}(t))w^*_{\lambda}(t)^2}
%\right)x} \ee+\Ordo(\lambda)
\end{proof}
 The r.h.s. of (\ref{limit_thm_Laplace}) is the Laplace transform of
 the $\Gamma(\frac12, 1)$ distribution
 and the r.h.s. of
 (\ref{subcrit_limit_thm_distribution_formulation}) is the distribution function
  of the $\Gamma(\frac12, 1)$ distribution, so
   \eqref{subcrit_limit_thm_distribution_formulation} follows from the
continuity theorem of Laplace transforms.
\begin{proof}[Proof of Theorem
\ref{subcrit_converges_to_crit}] First observe that instead of
proving uniform convergence of $\Phi_n$ to $\Phi_{crit}$ we only
need to show convergence on $\lbrack 0, T \rbrack$ for any $T$,
because \[T \geq \gt \; \implies \;
m_0(T)=\int_{T+w^*(T)}^{\infty} E(0,w)dw \leq \int_{T}^{\infty}
\frac{1}{w^2}dw =\frac{1}{T}\] by (\ref{upper_lower_bound_on_E}),
thus $0 \leq \Phi_n(t)-\Phi_{crit}(t)\leq \frac{1}{T}$ for $t\geq
T$. If we prove that $w^*(t)$ is small for $t \geq
\frac{1}{m_1(0)}$ then we are done by (\ref{Phi_and_F}) and Lemma
\ref{crit_eq_existence}, since \be
\label{if_w_star_is_small_then_Phi_is_close_to_Phi_crit} 0\leq
\Phi_n(t)-\Phi_{crit}(t)= F(0,t+w_n^*(t))-F(0,t) \leq w_n^*(t)
\phisup \quad t\geq \gt \ee \[\Phi_n(t)\leq
\Phi_n(\gt)=F(0,\gt+w_n^*(\gt)) \leq w_n^*(\gt) \phisup \quad t
\leq \gt \]

 We can give an upper bound on $w^*(t)$ for $t\geq \gt$
if we replace $\lambda(t)$ with $\lambda_{sup}$ in
(\ref{subcritical_control_ODE}): using
(\ref{upper_lower_bound_of_sc_control_ode_with_lambert}) we get
$w^*(t)=\Ordo \left( \lambda \log(\frac{1}{\lambda})\right)$ if we
substitute $t \geq \gt$ and $c=\frac{\lambda_{sup}}{\phiinf}$
 into (\ref{Lambert_megoldokeplet}), thus $\lim_{n \to \infty}w_n^*(t)=0$
  uniformly for $ \gt \leq t \leq T$.

We obtain $\lim_{n \to \infty} v_k^n(t)=v_k(t)$ for $k=1,2,\dots$
by the uniform convergence of $m_0^n(t)$ and $\lambda^n(t)$  to
the critical $m_0(t)=m_0(0)-\Phi(t)$ and $\lambda(t)\equiv 0$ in
(\ref{subcrit_smol_frozen_integal_eq}).
\end{proof}

\section{Proof of the alternating limit
theorem}\label{section_proof_of_alt}

We turn our attention to the proof of Theorem
\ref{alternating_converges_to_crit}. and Theorem
\ref{alter_limit_thm}.

In this section we assume $m_0(0)=1$ but the results generalize
easily to the $m_0(0)\neq 1$ case, since if $\conf{v}(t)$ is the
solution of \eqref{alter_smol_frozen_eq} with burning times
$\bt_1,\bt_2,\dots$ then $m_0(0)\conf{v}(m_0(0)t)$ is also a
solution of \eqref{alter_smol_frozen_eq} with burning times
$\frac{\bt_1}{m_0(0)},\frac{\bt_2}{m_0(0)},\dots$

\begin{definition}
If $\conf{v}(t)$ is a solution of
 (\ref{alter_smol_frozen_eq}), let $w^*_+(t):=\frac{1}{m_1(t)}$.

 If $w^*(t)\geq 0$ then $w^*_+(t)=w^*(t)$, but if
 $w^*(t)<0$ then $w^*_+(t)=-X'(t,-\theta(t))$.
\end{definition}
If $t$ is a burning time then $w^*(t_+):=\lim_{\varepsilon \to 0}w^*(t+\varepsilon)=  w^*_+(t)$.

\begin{lemma}\label{lemma_alter_bounds}
We consider the solution of (\ref{alter_smol_frozen_eq}) on
 $\lbrack 0, T \rbrack$  with an arbitrary sequence of
 burning times. If $\gt \leq t \leq T$ and $w^*(t)<0$ then
\be \label{global_theta_lower_bound} \theta(t) \geq
\frac{m_1(0)}{m_2(0)} \frac{1}{T^2} \abs{w^*(t)}
 \ee
 \be
\label{global_alter_wstar_bound}  w^*_+(t)\leq 4 \sqrt{\frac{m_2(0)}{m_1(0)}}
 \exp \left(\frac{m_2(0)}{m_1(0)}T+1 \right)\cdot
 \abs{w^*(t)}=:C(T,\conf{v}(0))\abs{w^*(t)}
 \ee
 If  $w^*(t)<0$ and if \be
 \label{condition_w1_w2}
  w_1 +\abs{w^*(t)} \leq t
 \leq
 w_2-\sqrt{\frac{\phisup(w_1,w_2)}{\phiinf(w_1,w_2)}}\abs{w^*(t)}\ee
holds then \be \label{alternating_flip_explicit_upper_lower_bound}
- \sqrt{\frac{\phiinf(w_1,w_2)}{\phisup(w_1,w_2)}}w^*(t)\leq
w^*_+(t) \leq -
\sqrt{\frac{\phisup(w_1,w_2)}{\phiinf(w_1,w_2)}}w^*(t) \ee
 \be \label{upper_lower_bound_w_star_theta} -2 \phiinf(w_1,w_2)w^*(t) \leq
\theta(t)\leq- 2 \phisup(w_1,w_2)w^*(t)  \ee
\end{lemma}

\begin{proof}
 By (\ref{defF}), $w^*_+(t)=-X'(t,-\theta(t))$,
(\ref{E_and_w_star_determines_everything}) and
\eqref{upper_lower_bound_on_E} we get \[ \theta(t)=F(t,w^*_+(t))
\geq \int_{w^*(t)}^0 \frac{m_1(0)}{m_2(0)} \frac{1}{(t+y)^2}dy
\geq \frac{m_1(0)}{m_2(0)} \frac{1}{T^2} \abs{w^*(t)} \]

Rearranging (\ref{defG}) and using (\ref{defF}) we get that
 $w=w^*_+(t)$ is the positive root of
the function
\[f(w):=G(t,w)-F(t, w)w=
G(t,0)+\left(-\int_0^w yE(t,y)\right)
dy=f(0)+\left(f(w)-f(0)\right) \] We prove
\eqref{global_alter_wstar_bound} by considering the cases
$\frac{\abs{w^*(t)}}{t} \leq
\frac{1}{4}\sqrt{\frac{m_1(0)}{m_2(0)}}$ and
$\frac{\abs{w^*(t)}}{t} > \frac{1}{4}\sqrt{\frac{m_1(0)}{m_2(0)}}$
separately.

If $\frac{\abs{w^*(t)}}{t} \leq
\frac{1}{4}\sqrt{\frac{m_1(0)}{m_2(0)}}$, then we prove that
$w^*_+(t) \leq 2 \sqrt{\frac{m_2(0)}{m_1(0)}} \abs{w^*(t)}$ by
showing that $f(0) \leq \abs{f(w)-f(0)}$ with $w= 2
\sqrt{\frac{m_2(0)}{m_1(0)}} \abs{w^*(t)}$.

\[f(0)=\int_{w^*(t)}^0 (-y)E(0,t+y)dy \leq \int_0^{\abs{w^*(t)}} \frac{y}{(t-y)^2} dy\]
by \eqref{E_and_w_star_determines_everything} and
\eqref{upper_lower_bound_on_E}.
\be\label{fwfnull_lower_integral_formula} \abs{f(w)-f(0)} \geq
\int_0^w \frac{m_1(0)}{m_2(0)} \frac{y}{(t+y)^2}dy
=\int_0^{\abs{w^*(t)}} \frac{m_1(0)}{m_2(0)}
\frac{y}{(t\frac{\abs{w^*(t)}}{w} +y)^2}dy \ee It is
straightforward to check that
\[ 0 \leq y \leq \abs{w^*(t)} \,\, \& \,\, \frac{\abs{w^*(t)}}{t} \leq
\frac{1}{4}\sqrt{\frac{m_1(0)}{m_2(0)}} \quad \implies \quad
\frac{y}{(t-y)^2} \leq \frac{m_1(0)}{m_2(0)}
\frac{y}{(t\frac{\abs{w^*(t)}}{w}+y)^2}\] which is sufficient for
$f(0) \leq \abs{f(2 \sqrt{\frac{m_2(0)}{m_1(0)}}
\abs{w^*(t)})-f(0)}$ to hold.

If $\frac{\abs{w^*(t)}}{t} >
\frac{1}{4}\sqrt{\frac{m_1(0)}{m_2(0)}}$, then \[f(0)=G(t,0)=
\int_{w^*(t)}^0 F(t,y)dy \leq \abs{w^*(t)} \leq T\] since by
\eqref{F_U_E_U_formulas} we have $F(t,y)\leq m_0(t)\leq m_0(0)=1$.
Calculating the middle integral in
\eqref{fwfnull_lower_integral_formula} we get that in order to
have $f(0) \leq \abs{f(w)-f(0)}$
\[\frac{m_1(0)}{m_2(0)}\left(\log(1+\frac{w}{t})-1\right)
\geq T \] is sufficient. Rearranging this and using
$\frac{\abs{w^*(t)}}{t} > \frac{1}{4}\sqrt{\frac{m_1(0)}{m_2(0)}}$
 we obtain \eqref{global_alter_wstar_bound}.

 The proof of the upper bound of
(\ref{alternating_flip_explicit_upper_lower_bound}) is similar:
using \eqref{E_and_w_star_determines_everything} we get that $w_1
\leq t- \abs{w^*(t)} \leq t+w \leq w_2$ implies
\[ f(0) \leq \frac{1}{2} \phisup(w_1,w_2) w^*(t)^2, \quad
f(w)-f(0) \leq -\frac12 \phiinf(w_1,w_2) w^2\] Using
\eqref{condition_w1_w2}
 the
inequality
 $f \left(
 -\sqrt{\frac{\phisup(w_1,w_2)}{\phiinf(w_1,w_2)}}w^*(t)
 \right) \leq 0$ follows.
 The lower bound of (\ref{alternating_flip_explicit_upper_lower_bound}) is verified similarly.

If $u \in \lbrack -\theta(t) ,0 \rbrack$, then
 \[X(t,u)\leq -w^*(t)u +\frac12 \frac{1}{\phiinf(w_1,w_2)} u^2,\]
 since $\kderi{X}{u}(t,u)$ with $u \in \lbrack -\theta(t) ,0
\rbrack$ is equal to $\frac{1}{E(0,t+y)}$ for some \[y \in \lbrack
w^*(t), w^*_+(t) \rbrack \subseteq \lbrack w^*(t),
-\sqrt{\frac{\phisup(w_1,w_2)}{\phiinf(w_1,w_2)}}w^*(t) \rbrack,\]
thus $t+y \in \lbrack w_1,w_2 \rbrack$ by \eqref{condition_w1_w2}.
This implies the lower bound of
(\ref{upper_lower_bound_w_star_theta}), and the proof of the upper
bound is similar.

\end{proof}
The proof of Theorem \ref{alternating_converges_to_crit}. is
similar that of Theorem \ref{subcrit_converges_to_crit}.: if
$\varepsilon=\sup_i \{ \bt_{i+1}-\bt_{i}\}$ and $\bt_i < t \leq
\bt_{i+1}$ then $w^*(t)=w^*_+(\bt_i)-(t-\bt_i) \geq -\varepsilon$
and by \eqref{global_alter_wstar_bound} we have $w^*_+(\bt_i)
=\Ordo(\abs{w^*(\bt_i}))=\Ordo(\varepsilon)$ on $\lbrack 0, T
\rbrack$.

% (\ref{lipschitz_phisup_phiinf}) and
%(\ref{alternating_flip_explicit_upper_lower_bound}) implies
%$\abs{w^*(t)}\leq \varepsilon +\Ordo(\varepsilon^2)$ for all $\gt
%\leq t \leq T$.

\begin{lemma}\label{lemma_burning_time_in_between_gel_times}
We consider the solution of (\ref{alter_smol_frozen_eq}) with
initial critical core $E(0,w)$. If $\gt_1<\gt_2$ are two
consecutive gelation times,
 then the unique burning time in between $\gt_1$ and $\gt_2$ is
\be \label{burning_time_in_between_gel_times}
\bt(\gt_1,\gt_2) = \frac{\int_{\gt_1}^{\gt_2} y E(0,y)dy}{\int_{\gt_1}^{\gt_2} E(0,y)dy}
 \ee
Moreover \be \label{thata_at_burningtime}
\theta(\bt(\gt_1,\gt_2))= \int_{\gt_1}^{\gt_2} E(0,y)dy\ee

\end{lemma}
\begin{proof}
$\bt$ needs to satisfy $\gt_2-\bt=w^*_+(\bt)$, but by the proof of
Lemma \ref{lemma_alter_bounds}.
 $w^*_+(\bt)$ is the unique positive root of $G(\bt,0)-\int_0^w y E(0,\bt+y)dy$.
$G(\bt,0)=-\int_{\gt_1-\bt}^0 y E(0,\bt+y)dy$ by
(\ref{E_and_w_star_determines_everything}), so
$\int_{\gt_1-\bt}^{\gt_2-\bt} y E(0,\bt+y)dy=0$ must hold,
from which  (\ref{burning_time_in_between_gel_times}) easily follows.

By (\ref{defF}), $w^*_+(t)=-X'(t,-\theta(t))$ and
(\ref{E_and_w_star_determines_everything}) we get
\[\theta(\bt)=F(\bt,w^*_+(\bt))=\int_{w^*(\bt)}^{w^*_+(\bt)} E(0,\bt +y)dy=
\int_{\gt_1}^{\gt_2} E(0,y)dy  \]

\end{proof}

\begin{definition}\label{def_of_things_needed_for_alter_lim_thm}
 If $\conf{v}(t)$ is the solution of
the random alternating equations (see Definition
\ref{def_random_alternating_eqs}.), denote by $\bt_1<\bt_2<\dots$
the sequence of random burning times and by
$\gt=\gt_1<\gt_2<\dots$ the sequence of random gelation times.
Indeed $\gt_1<\bt_1<\gt_2<\bt_2<\dots$

 Let
$\tau_i:=\gt_{i+1}-\gt_i$ be the length of the $i$-th critical
interval.
 \[N(t):=\max\{ i: \gt_i<t \}, \quad \tau(t,i):=\tau_{N(t)+i}\]
$\tau(t,0)$ is the length of the critical interval containing $t$.

 Let $\theta(t,i):=\theta(\bt_{N(t)+i})$, thus $\theta(t,1)$
is the frozen mass of the first giant component born after $t$.
\[w^*_{-}(t,i):=\bt_{N(t)+i}-\gt_{N(t)+i}=-w^*(\bt_{N(t)+i})\]
 \[w^*_{+}(t,i):=\gt_{N(t)+i+1}-\bt_{N(t)+i}=w^*_+(\bt_{N(t)+i})\]
\end{definition}

\begin{definition} \label{Rayleigh}
A nonnegative random variable $X$ has Rayleigh distribution with
parameter $\sigma$, briefly $X \sim R(\sigma)$, if \[\prob{X
>x}=\exp(-\frac{1}{2 \sigma^2} x^2)=:R(\sigma,x)\] $\expect{X}=\sigma
\sqrt{\frac{\pi}{2}}$. $Y$ has a size-biased Rayleigh distribution
with parameter $\sigma$, briefly $Y \sim R_{sb}(\sigma)$ if
\[\prob{Y>y}=\frac{\expect{X \cdot \ind \lbrack X
>y \rbrack}}{\expect{X}}=R_{sb}(\sigma,y) \]
The scaling identities \be \label{Rayleigh_scaling}
R(\sigma,x)=R(a\sigma,ax) \quad \text{and} \quad
R_{sb}(\sigma,x)=R_{sb}(a\sigma,ax)\ee are valid for $a>0$.
\end{definition}

The r.h.s of (\ref{alter_limit_thm_statement_formula}) is
$R_{sb}(\frac{1}{\sqrt{2}},x)$.

The Rayleigh distribution emerges in our setting in the following
way: if we consider the solution of the random alternating
equations with burning times defined by a homogenous Poisson
process with rate $\lambda$, forget about the error terms in
(\ref{upper_lower_bound_w_star_theta}) by assuming $w_1=w_2$
 then $\theta(t)=2E \cdot
(t-\gt_i)$ if $\gt_i < t \leq \bt_i$, so \[\prob{\bt_i -\gt_i
>w}=\exp(-\lambda \int_0^w 2Es ds)=R(\frac{1}{\sqrt{2E\lambda}},w). \]
From $\theta(\bt_{i})=2E\cdot(\bt_i-\gt_i)$ and
\eqref{Rayleigh_scaling} we get
 $\theta(\bt_{i})
\sim R(\sqrt{\frac{2 E}{\lambda}})$. Assuming $w_1=w_2$ in
\eqref{alternating_flip_explicit_upper_lower_bound} we get
$w^*_+(\bt_i)=-w^*(\bt_i)$, thus
 $\tau_i \sim
R(\sqrt{\frac{2}{E\lambda}})$.

\begin{lemma}\label{bound_on_the distance of_critical times}
If $\conf{v}(t)$ is the solution of the random alternating
equations with constant rate function $\lambda(t)\equiv \lambda$
  then for every $\gt \leq t \leq T$ we have
  \be\label{maradek_theta_becs}
   \expect{ \theta(\bt_{N(t)})\ind\lbrack  \bt_{N(t)}<T \rbrack }
  =\Ordo(\lambda^{-\frac12})\ee
\be\label{maradek_ido_becs} \expect{ \gt_{N(t)+1}\wedge T -t}
=\Ordo(\lambda^{-\frac12})\ee
 as $\lambda \to \infty$ where the
constant in the $\Ordo$ depends only on the initial data and $T$.
\end{lemma}
\begin{proof}
Let $\gamma(t):=t-\gt_{N(t)}$. Then
\begin{multline*}
\lim_{dt \to 0} \frac{1}{dt} \expect{\gamma(t+dt)-\gamma(t) \mycond
\mathcal{F}_t}=1-\gamma(t)\lim_{dt \to 0} \frac{1}{dt} \prob{ t \leq \gt_{N(t)+1} \leq  t+dt  \mycond \mathcal{F}_t}=\\
1-\gamma(t)\lim_{dt \to 0} \frac{1}{dt} \prob{ \bt(\gt_{N(t)},t) \leq \bt_{N(t)} \leq  \bt(\gt_{N(t)},t+dt)   \mycond \gamma(t)}=\\
1-\gamma(t)  \theta( \bt(\gt_{N(t)},t)) \lambda \left. \frac{d}{ds} \bt(\gt_{N(t)},s) \right|_{s=t}=\\
1-\lambda E(0,t) \gamma(t) \left(t-\bt(t-\gamma(t),t)\right) \leq
1- \frac12 \lambda \frac{\phiinf(T)^2}{\phisup}\gamma(t)^2
\end{multline*}
by Lemma \ref{lemma_burning_time_in_between_gel_times}.
Taking the expectation of both sides of the above inequality and applying
 Jensen's inequality we get \[\frac{d}{dt} \expect{\gamma(t)} \leq
1- \frac12 \lambda
\frac{\phiinf(T)^2}{\phisup}\expect{\gamma(t)}^2.\] This
differential inequality together with $\gamma(\gt)=0$ implies
\[\expect{\gamma(t)} \leq \frac{1}{\sqrt{\lambda}} \frac{\sqrt{2
\phisup}}{\phiinf(T)}=\Ordo(\lambda^{-\frac12})\quad \quad  \gt
\leq t \leq T.\] by a "forbidden region"-type argument. Now we
prove \be\label{fapipa}  \expect{ \gt_{N(t)+1}\wedge T -\gt_{N(t)}
} =\Ordo(\lambda^{-\frac12})\ee
 from which \eqref{maradek_ido_becs}
trivially follows. We obtain \eqref{maradek_theta_becs} using
\eqref{fapipa} and $\theta(\bt_{N(t)}) \leq 2 \phisup \cdot
(\bt_{N(t)}-\gt_{N(t)})$ by the upper bound of
\eqref{upper_lower_bound_w_star_theta}.
\begin{multline*}
\gt_{N(t)+1}\wedge T -\gt_{N(t)}= \gamma(t) + \left(\gt_{N(t)+1}
\wedge T -t \right) \ind \lbrack t \geq \bt_{N(t)} \rbrack + \\
 \left( \gt_{N(t)+1}
\wedge T -\bt_{N(t)}\wedge T \right)\ind \lbrack t <
\bt_{N(t)}\rbrack+ \left( \bt_{N(t)} \wedge T -t\right)\ind
\lbrack t < \bt_{N(t)}\rbrack
\end{multline*}
\begin{multline*}
\left(\gt_{N(t)+1}\wedge T -t\right)
 \ind \lbrack t \geq \bt_{N(t)} \rbrack \, \leq \,
 w^*_+(t,0) \ind \lbrack t \geq \bt_{N(t)}
\rbrack \, \leq \\
 C(T,\conf{v}(0)) w^*_-(t,0) \ind \lbrack t \geq
\bt_{N(t)} \rbrack \, \leq \,
 C(T,\conf{v}(0)) \gamma(t)
\end{multline*}
where $C(T,\conf{v}(0))$ is defined in
\eqref{global_alter_wstar_bound}.
\begin{multline*}
 \left( \gt_{N(t)+1}
\wedge T -\bt_{N(t)}\wedge T \right)\ind \lbrack t <
\bt_{N(t)}\rbrack \, \leq \, w^*_+(t,0) \ind \lbrack t <
\bt_{N(t)} \leq T \rbrack \, \leq \\ C(T,\conf{v}(0)) \gamma(t) +
C(T,\conf{v}(0))\left( \bt_{N(t)} \wedge T -t\right)\ind \lbrack t
< \bt_{N(t)}\rbrack
\end{multline*}
By \eqref{ramdom_alt_burning_time} and
\eqref{global_theta_lower_bound} we have
\begin{multline*}
\expect{\left( \bt_{N(t)} \wedge T -t\right)\ind \lbrack t <
\bt_{N(t)}\rbrack}=\expect{\left(\bt_{N(t)} \wedge T -t\right)\lor 0 }=\\
\int_0^{T-t} \prob{\bt_{N(t)}-t \geq x} dx \leq
\int_0^{T-t}\exp\left( -\lambda \int_0^x \frac{m_1(0)}{m_2(0)}
\frac{1}{T^2}y dy\right)=\Ordo( \lambda^{-\frac12})
\end{multline*}
\[ \expect{\gt_{N(t)+1}\wedge T
-\gt_{N(t)}}=\Ordo(\expect{\gamma(t)}) +\Ordo\left( \expect{\left(
\bt_{N(t)} \wedge T -t\right)\lor 0}
\right)=\Ordo(\lambda^{-\frac12})
\]
\end{proof}

\begin{proof}[Sketch proof of of Theorem \ref{alter_limit_thm}.]
 Our aim is to make the following argument rigorous:
Let \[n(\lambda):=\lfloor \epslam\sqrt{\frac{E(t,0)\lambda}{\pi}}
\rfloor.\]
  If
$1 \ll \lambda$ then
$\theta(t,1),\theta(t,2),\dots,\theta(t,n(\lambda))$ are "almost"
i.i.d. with distribution $\theta(t,i) \sim R(\sqrt{\frac{2
E(t,0)}{\lambda}})$. $\tau(t,i) \approx
\frac{\theta(t,i)}{E(t,0)}$, so \[\sum_{i=1}^{n(\lambda)}
\tau(t,i) \approx \epslam\] by the weak law of large numbers.
Substituting $\hat{x}=2\sqrt{\frac{E(t,0)}{\lambda}} x$ into
\[
\frac{\Phi\left(t+\epslam,\hat{x}
 \right)-
\Phi \left( t,\hat{x} \right)}{\epslam E(t,0)} \approx
\frac{\sum_{i=1}^{n(\lambda)}\theta(t,i)\cdot\ind \lbrack
\theta(t,i)>\hat{x}\rbrack}
{\sum_{i=1}^{n(\lambda)}\theta(t,i)}\approx \frac{
\expect{\theta(t,1)\ind \lbrack \theta(t,1)>\hat{x}\rbrack } }
{\expect{\theta(t,1)}}
\]
we get (\ref{alter_limit_thm_statement_formula}).
\end{proof}

\begin{proof}[ Proof of Theorem \ref{alter_limit_thm}.]
We use the notations of Definitions
\ref{def_of_things_needed_for_alter_lim_thm}. and \ref{Rayleigh}.
\[ E:=E(t,0)=E(0,t)=\varphi_{crit}(t)\]We fix
 $x \geq 0$ and define
  \[\hat{x}:=2\sqrt{\frac{E}{\lambda}} x, \quad
   \theta(t,i,\hat{x}):=\theta(t,i)\ind \lbrack \theta(t,i)>\hat{x}
\rbrack, \quad n(\lambda, z):=\lfloor
\epslam\sqrt{\frac{E\lambda}{\pi}} (1+z) \rfloor \]
 By the assumption
$\lambda^{-\frac12} \ll \epslam$ we have $\lim_{\lambda \to
\infty} n(\lambda,z)=+\infty$ for any $-1<z$.

 Let
$m(\lambda):=N(t+\epslam)-N(t)-1$.
\begin{multline}\label{rewriting_the_sum_of_burnt_giants}
 \Phi\left(t+\epslam,\hat{x}
 \right)-
\Phi \left( t,\hat{x} \right)= \theta(t,0,\hat{x}) \ind \lbrack
\bt_{N(t)} >t \rbrack +\\ \sum_{i=1}^{m(\lambda)}
\theta(t,i,\hat{x})+ \theta(t+\epslam,0,\hat{x}) \ind \lbrack
\bt_{N(t+\epslam)}<t+\epslam \rbrack
\end{multline}

In order to prove (\ref{alter_limit_thm_statement_formula}) we
only need to show that we have $\lim_{\lambda \to \infty} \prob{
B(\lambda, \varepsilon)}=1$ for every $\varepsilon>0$ where
 \[ B(\lambda, \varepsilon):= \left\{
R_{sb}(\frac{1}{\sqrt{2}},x)-\varepsilon
<\frac{\sum_{i=1}^{m(\lambda)} \theta(t,i,\hat{x})}{E \epslam} <
R_{sb}(\frac{1}{\sqrt{2}},x)+\varepsilon \right \} \] because the
first and the last term on the r.h.s. of
(\ref{rewriting_the_sum_of_burnt_giants}) divided by $E \epslam$
converge to $0$ in probability as $\lambda \to \infty$ by
\eqref{maradek_theta_becs} and $\lambda^{-\frac12} \ll \epslam$.

\[\phisup(\lambda):=\phisup(t,t+2\epslam), \quad
\phiinf(\lambda):=\phiinf(t,t+2\epslam)\] By
(\ref{lipschitz_phisup_phiinf}) we have \be
\label{phisup_and_phiinf_are_almost_E_if_lambda_is_big}
\phisup(\lambda) \leq E+2D \epslam \quad \text{and} \quad E-2D
\epslam \leq \phiinf(\lambda) \ee  \[C^u(\lambda):= 1+
   \sqrt{\frac{\phisup(\lambda)}{\phiinf(\lambda)}} \quad C^l(\lambda):= 1+
   \sqrt{\frac{\phiinf(\lambda)}{\phisup(\lambda)}}\]

   $\lim_{\lambda \to \infty} C^u(\lambda)=\lim_{\lambda \to \infty} C^l(\lambda)= 2$,
    since $\epslam \ll 1$.

 We are
going to couple the random variables $\gt_{N(t)+1},
w^*_-(t,1),w^*_-(t,2),\dots$ to \[w^l_-(1),w^l_-(2),\dots \quad \text{ and }
\quad w^u_-(1),w^u_-(2),\dots\] where
  $w^l_-(i) \sim R(\frac{1}{\sqrt{2\phisup(\lambda)\lambda}})$ are
  i.i.d. and $w^u_-(i) \sim R(\frac{1}{\sqrt{2\phiinf(\lambda) \lambda}})$ are
  i.i.d., moreover the auxiliary random variables are independent
  from $\gt_{N(t)+1}$. If we define the events
  \[ A^u (\lambda, z, z_2):=
   \left\{ \gt_{N(t)+1}+ C^u(\lambda) \cdot
   \sum_{j=1}^{n(\lambda,z)} w^u_-(j) \leq t +\epslam \cdot (1+z_2) \right\}\]
\[ A^l (\lambda, z, z_2):=
   \left\{ \gt_{N(t)+1}+ C^l(\lambda) \cdot
   \sum_{j=1}^{n(\lambda,z)} w^l_-(j) \geq t +\epslam \cdot (1+z_2)\right\} \]
then it is an easy consequence of \eqref{maradek_ido_becs},
$\lambda^{-\frac12} \ll \epslam $, and the weak law of large
numbers that \[ -1<z <z_2 \implies \lim_{\lambda \to \infty}
\prob{A^u (\lambda, z, z_2)}=1\] \[z
>z_2>-1 \implies \lim_{\lambda \to \infty} \prob{A^l
(\lambda, z, z_2)}=1 \] Our coupling is going to satisfy
   \be \label{coupling_inequalities} A^u (\lambda, z, 1) \subseteq
  \bigcap_{i=1}^{n(\lambda,z)}  \{ w^l_-(i) \leq w^*_-(t,i)\leq
  w^u_-(i) \}
 \ee
for any $z$.

The joint construction of $w^l_-(j)$,  $w^*_-(t,j)$ and
  $w^u_-(j)$ for $j=1,2,\dots$ is as follows:
 given $\gt_{N(t)+1}$ and $w^*_-(t,1), \dots, w^*_-(t,j-1)$ we
can determine
 $\gt_{N(t)+j}$ by solving (\ref{alter_smol_frozen_eq}). For $s \geq
0$ Let \[\mu(s):=\lambda \theta(\gt_{N(t)+j}+s),\quad
\mu_l(s):=\lambda 2 \phisup(\lambda) s,\quad
 \mu_u(s):=\lambda 2
\phiinf(\lambda) s.\] Let $w^*_-(t,j)$, $w^l_-(j)$ and $w^u_-(j)$
be the horizontal coordinate of the leftmost point below the
graphs of $\mu$, $\mu_l$ and $\mu_u$ of the same standard uniform
2-dimensional Poisson process on the first quadrant of the plane.
Thus $w^l_-(j) \sim R(\frac{1}{\sqrt{2\phisup(\lambda)\lambda}})$,
 $w^u_-(j) \sim R(\frac{1}{\sqrt{2\phiinf(\lambda)\lambda}})$ are
independent from everything that was constructed earlier and
$\prob{w^l_-(j)\leq w^u_-(j)}=1$. The joint distribution of
$\gt_{N(t)+1}, w^*_-(t,1),\dots, w^*_-(t,j)$ agrees with that of
the solution of the random alternating equation.

We are going to prove (\ref{coupling_inequalities}) by induction.
Assume that $A^u (\lambda, z, 1)$ holds. If \be
\label{induction_hypothesis} \bigcap_{i=1}^{j-1} \{ w^l_-(i) \leq
w^*_-(t,i)\leq
  w^u_-(i) \}\, \cap \, \bigcap_{i=1}^{j-1} \{ \tau(t,i) \leq C^u(\lambda)\cdot
  w^u_-(i) \} \ee  holds for some $j \leq
  n(\lambda,z)$, then
  \[\gt_{N(t)+j} = \gt_{N(t)+1} +\sum_{i=1}^{j-1} \tau(t,i) \leq \gt_{N(t)+1} +C^u(\lambda)\sum_{i=1}^{j-1}w^u_-(i)  \]
which implies $\mu_u(s) \leq \mu(s) \leq \mu_l(s)$ for $0\leq s
\leq w^u_-(j)$ by (\ref{upper_lower_bound_w_star_theta}) and $A^u
(\lambda, z, 1)$. From this  $w^l_-(j) \leq w^*_-(t,j)\leq
  w^u_-(j)$ follows, and (\ref{alternating_flip_explicit_upper_lower_bound}) can be
   applied
  to deduce
   \[\tau(t,j) =w^*_-(t,j)+w^*_+(t,j) \leq \left(1+
   \sqrt{\frac{\phisup(\lambda)}{\phiinf(\lambda)}}\right)w^*_-(t,j)
\leq C^u(\lambda)w^u_-(j)\] Thus we can replace $j$ with $j+1$ in
\eqref{induction_hypothesis}. This completes the proof of
(\ref{coupling_inequalities}). Let
\[\theta^u(t,i,\hat{x}):=2 \phisup(\lambda) w^u(i) \cdot \ind
\lbrack 2 \phisup(\lambda) w^u(i)
> \hat{x} \rbrack\] \[\theta^l(t,i,\hat{x}):=2 \phiinf(\lambda)
w^l(i) \cdot \ind \lbrack 2 \phiinf(\lambda) w^l(i) > \hat{x}
\rbrack\]

(\ref{upper_lower_bound_w_star_theta}) and
(\ref{coupling_inequalities}) imply

\[ A^u (\lambda, z, 1)  \subseteq
 \bigcap_{i=1}^{n(\lambda,z)}  \{ \theta^l(t,i,\hat{x}) \leq \theta(t,i,\hat{x}) \leq
  \theta^u(t,i,\hat{x}) \}
\]
 \[ B^u (\lambda, z, \varepsilon):=
   \left\{
   \frac{\sum_{i=1}^{n(\lambda,z)} \theta^u(t,i,\hat{x})}{E \epslam} \leq R_{sb}(\frac{1}{\sqrt{2}},x) +\varepsilon \right\}\]
\[ B^l (\lambda, z, \varepsilon):=
   \left\{
   \frac{\sum_{i=1}^{n(\lambda,z)} \theta^l(t,i,\hat{x})}{E \epslam} \geq R_{sb}(\frac{1}{\sqrt{2}},x) -\varepsilon \right\}\]
The law of large numbers, \eqref{Rayleigh_scaling} and
(\ref{phisup_and_phiinf_are_almost_E_if_lambda_is_big}) imply that
\[ z<\varepsilon \; \implies \;  \lim_{\lambda \to \infty} \prob{B^u
(\lambda, z, \varepsilon)}=1\quad \text{ and }\quad   -\varepsilon< z  \implies
\lim_{\lambda \to \infty} \prob{B^l (\lambda, z,
\varepsilon)}=1.\]

 We can use
(\ref{coupling_inequalities}) and
\eqref{alternating_flip_explicit_upper_lower_bound} to show
\[A^u (\lambda, z, 1) \subseteq \bigcap_{i=1}^{n(\lambda,z)} \{
C^l(\lambda) w^l_-(i) \leq \tau(t,i) \leq C^u(\lambda) w^u_-(i),
\}
\]
 Since \[m(\lambda)=\max \{j:
\gt_{N(t)+1}+\sum_{i=1}^j \tau(t,i) < t+\epslam \}\] and
$A^u(\lambda,z,0) \subseteq A^u(\lambda,z,1)$ by definition,
\[ A^u(\lambda,z,0)
\subseteq \{ m(\lambda) \geq n(\lambda,z) \}, \quad A^l(\lambda,
z, 0) \cap A^u(\lambda,z,1) \subseteq \{ m(\lambda) \leq
n(\lambda,z) \},\]  \begin{multline*}
A^l(\lambda,\frac{\varepsilon}{2},0) \cap
A^u(\lambda,\frac{\varepsilon}{2},1) \cap
B^u(\lambda,\frac{\varepsilon}{2},\varepsilon) \cap
B^l(\lambda,-\frac{\varepsilon}{2}, \varepsilon) \cap
A^u(\lambda,-\frac{\varepsilon}{2},0) \subseteq
B(\lambda,\varepsilon) \end{multline*} This completes the proof of
$\lim_{\lambda \to \infty} \prob{ B(\lambda, \varepsilon)}=1$.

\end{proof}
$ $

{\bf Acknowledgements:} I thank
B\'alint T\'oth for introducing me to the topic of forest fire and
frozen percolation models, and the anonymous referees, whose comments helped me improving the paper.
 This research was  partially supported
by the OTKA (Hungarian National Research Fund) grant K 60708.

\end{document}